\documentclass[12pt,centertags,twoside]{amsart}
\usepackage{amsmath,amstext,amsthm,a4,amssymb,amscd}
\usepackage[mathscr]{eucal}
\usepackage{mathrsfs}
\usepackage{epsf}
\usepackage{typearea}
\usepackage{commath}
\usepackage{dsfont}
\usepackage{charter}
\usepackage[T1]{fontenc}
\usepackage{enumitem}
\setlist[1]{itemsep=5pt}

\usepackage[a4paper,width=16.2cm,top=3cm,bottom=3cm]
{geometry}

\numberwithin{equation}{section}

\newcommand{\field}[1]{\mathbb{#1}}
\newcommand{\Z}{\field{Z}}
\newcommand{\R}{\field{R}}
\newcommand{\C}{\field{C}}
\newcommand{\N}{\field{N}}

 \def\cC{\mathscr{C}}
\def\cL{\mathscr{L}}

\def\cO{\mathscr{O}}

\def\mC{\mathcal{C}}
\def\mA{\mathcal{A}}
\def\mB{\mathcal{B}}

\def\mO{\mathcal{O}}

\newcommand{\rcal}{\mathcal{R}}

\def\bE{{\boldsymbol E}}

\def\Re{{\rm Re}}
\def\Im{{\rm Im}}

\newcommand{\be}{\begin{eqnarray}}
\newcommand{\ee}{\end{eqnarray}}
\newcommand{\ov}{\overline}

\newcommand{\var}{\varepsilon}

\newcommand{\half}{{\frac{1}{2}}}

\DeclareMathOperator{\End}{End}

\DeclareMathOperator{\rank}{rk}
\DeclareMathOperator{\Id}{Id}

\DeclareMathOperator{\tr}{Tr}

\newcommand{\om}{\omega}
\newcommand{\db}{\overline\partial}
\newtheorem{thm}{Theorem}[section]
\newtheorem{lemma}[thm]{Lemma}
\newtheorem{prop}[thm]{Proposition}
\newtheorem{cor}[thm]{Corollary}
\theoremstyle{definition}
\newtheorem{rem}[thm]{Remark}
\theoremstyle{definition}

\newcommand{\comment}[1]{}

\begin{document}

\title
{Scaling asymptotics of heat kernels of
line bundles}

\author{Xiaonan Ma}
\address{Institut Universitaire de France
\&Universit\'e Paris Diderot - Paris 7,
UFR de Math\'ematiques, Case 7012,
75205 Paris Cedex 13, France}
\email{xiaonan.ma@imj-prg.fr}
\thanks{X.\ M.\ partially supported by
Institut Universitaire de France and
funded through the Institutional Strategy of
the University of Cologne within the German Excellence Initiative}
\author{George Marinescu}
\address{Universit{\"a}t zu K{\"o}ln,
Mathematisches Institut, Weyertal 86-90, 50931 K{\"o}ln, Germany\\
    \& Institute of Mathematics `Simion Stoilow', Romanian Academy,
Bucharest, Romania}
\email{gmarines@math.uni-koeln.de}
\thanks{G.\ M.\ partially supported by DFG funded
projects SFB/TR 12, MA 2469/2-1 and ENS Paris}

\author{Steve Zelditch}
\address{Department of Mathematics, Northwestern  University,
Evanston, IL 60208, USA}

\email{zelditch@math.northwestern.edu}

\thanks{S.\ Z.\ partially supported by NSF grant DMS-1206527.}

\dedicatory{Dedicated to Professor Duong H. Phong on the
occasion of his 60th birthday}

\begin{abstract}
We consider a general Hermitian
holomorphic line bundle $L$ on a compact complex manifold $M$
and let ${\Box}^q_p$ be the Kodaira Laplacian on $(0,q)$ forms
with values in  $L^p$. We study the scaling asymptotics
of the heat kernel $\exp(-u {\Box}^q_p/p)(x,y)$.

The main result is a complete
asymptotic expansion  for the semi-classically scaled heat kernel
$\exp(-u{\Box}^q_p/p)(x,x)$ along the diagonal.
It is a generalization of the Bergman/Szeg\"o kernel asymptotics
in the case of a positive line bundle, but no positivity is assumed.
We give two proofs, one based on the Hadamard parametrix for
the heat kernel on a principal bundle and the second based
on the analytic localization of the Dirac-Dolbeault operator.
%
\end{abstract}

\maketitle
\tableofcontents

\section{Introduction}
Let  $(M,J)$  be a complex manifold with complex structure $J$,
and complex dimension $n$.
Let $L$ and $E$ be two holomorphic vector bundles on
$M$ such that $\rank(L)=1$; the bundle $E$ plays the role of  an
auxiliary twisting bundle.
We fix Hermitian metrics $h^L$, $h^E$ on $L$, $E$.
Let $L^p$ denote the $p$th tensor power of $L$.
The purpose of this article is to prove  scaling
asymptotics of various heat kernels on  $L^p \otimes E$
as $p \to \infty$. We present the scaling asymptotics
from two points of view. The first one (Theorem \ref{ASYM})
gives scaling asymptotics of the Kodaira heat kernels
and is based on the analytic localization technique
of Bismut-Lebeau \cite{BL},  adaptating the arguments
from \cite[\S 1.6, \S 4.2]{MM07}.
The second (Theorem \ref{ASYMX}) gives scaling asymptotics
of the heat kernels associated to the Bochner Laplacian,
and is  an adaptation of the Szeg\"o
kernel asymptotics of \cite{Z}. It   is based on  lifting  sections
of $L^p$ to equivariant functions  on  the
associated  principal $S^1$ bundle
$X_h \to M$, and obtaining  scaling asymptotics of heat kernels
from  Fourier analysis of characters and stationary
phase asymptotics. Either method can be applied to any of the
relevant heat kernels and it seems
to us of some interest to compare the methods.
We refer to \cite{Bis.V, Z,MM07,Z2} for background from both
points of view  of analysis on higher powers of line bundles.

To state our results, we need to introduce some notation.
Let $\nabla ^E$, $\nabla ^L$ be the holomorphic Hermitian
connections on $(E,h^E)$, $(L,h^L)$.
Let $R^L$, $R^E$
be the curvatures of $\nabla ^L$, $\nabla ^E$.
Let $g^{TM}$ be a $J$-invariant Riemannian metric on $M$,
i.\,e., $g^{TM}(Ju,Jv)= g^{TM}(u,v)$ for all $x\in M$ and
$u,v\in T_xM$. Set
\begin{equation}\label{lm4.1}
\omega : = \frac{\sqrt{-1}}{2\pi}R^L, \qquad
\Theta (\cdot, \cdot):= g^{TM}(J\cdot, \cdot).
\end{equation}
Then $\omega, \Theta$ are real $(1,1)$-forms on $M$,
and $\omega$ is the Chern-Weil representative of
the first Chern class \index{Chern class!first} $c_1(L)$ of $L$.
The Riemannian volume form  $dv_{M}$ of $(TM, g^{TM})$
is $\Theta^n/n !$.
We will identify the 2-form $R^L$ with the Hermitian matrix
$\dot{R}^L \in \End(T^{(1,0)}M)$ defined by
\begin{equation}\label{lm4.2}
\big\langle \dot{R}^LW,\ov{Y}\,\big\rangle=R^L (W,\ov{Y})\,,
\quad W,Y\in T^{(1,0)}M.
\end{equation}
The curvature $R^L$ acts as a derivation $\om_d\in \End(\Lambda (T^{*(0,1)}M))$  
on $\Lambda (T^{*(0,1)}M)$.
Namely, let $\{w_j\}_{j=1}^n$ be a local orthonormal frame of $T^{(1,0)}M$
with dual frame $\{w^j\}_{j=1}^n$. Set
\begin{equation}  \label{lm4.3}
\om_d=-\sum_{l,m} R^L (w_l,\overline{w}_m)\,
\overline{w}^m\wedge
\,i_{\overline{w}_l}\,,
\qquad \tau(x)=\sum_j R^L (w_j,\overline{w}_j)\,.
\end{equation}
Consider the Dolbeault-Dirac operator
\begin{align}\label{lm2.1}
D_p = \sqrt{2}\left(\overline{\partial}^{L^p\otimes E}
+ \,\overline{\partial}^{L^p\otimes E,*}\right)\,,
\end{align}
and the Kodaira Laplacian
\begin{align}\label{lm2.2}
\square_p=\tfrac12 D^2_p= \overline{\partial}^{L^p\otimes E}\,
\overline{\partial}^{L^p\otimes E,*}
+\,\overline{\partial}^{L^p\otimes E,*}\,
\overline{\partial}^{L^p\otimes E}\,.
\end{align}
For $p\in \N$, we denote by\index{$E_p$}
\begin{align}\label{lm4.7}
E_p^j:=\Lambda^j(T^{*(0,1)}M)\otimes L^p \otimes E,
\quad E_p=\oplus_j E_p^j\,,
\end{align}
and let $h_p$ the induced Hermitian metric on $E_p$\,.

The operator $D_p^2= 2\square_p$ is a second order
elliptic differential operator with principal symbol
$\sigma(D_p^2) (\xi)= |\xi|^2\Id_{E_p}$ for $\xi\in T_{x} ^*M$, $x\in M$.
The heat operator
$\exp(-uD_p^2)$ is well defined for $u>0$.
Let $\exp(-uD_p^2)(x,x')$, where $x,x'\in M$, be its smooth
kernel with respect
to the Riemannian volume form $dv_{M} (x')$. Then
\begin{align}\label{lm4.9}
\exp(-uD_p^2)(x,x')\in (E_p)_x
\otimes (E_p)^*_{x'}\,,
\end{align}
especially
\begin{align}\label{lm4.10}
\exp(-uD_p^2)(x,x) \in
\End(E_p)_x
= \End(\Lambda (T^{*(0,1)}M) \otimes E)_x,
\end{align}
where we use the canonical identification $\End(L^p)=\C$
for any line bundle $L$ on $M$.
Note that $D_p^2$ preserves the $\Z$-grading of
the Dolbeault complex
$\Omega^{0,\bullet}(M, L^p\otimes E)$, so
\begin{align}\label{lm4.91}
\exp(-uD_p^2|_{\Omega^{0,j}})(x,x')
=\sum_{k=1}^\infty e^{-u\lambda^j_{k,p}}
\varphi^j_{k,p}(x)\otimes\varphi^j_{k,p}(x')^*\in (E^j_p)_x
\otimes (E^j_p)^*_{x'}\,,
\end{align}
where $\{\lambda^j_{k,p}: k\geqslant1\}$ is the spectrum of
$D_p^2|_{\Omega^{0,j}}$ and $\{\varphi_{k,p}^j: k\geqslant1\}$
is an orthonormal basis of $L^2(M,E^j_p)$ consisting of
eigensections of $D_p^2|_{\Omega^{0,j}}$
with $D_p^2\varphi^j_{k,p}=\lambda^j_{k,p}\varphi^j_{k,p}$\,,
cf.\ \cite[(D.1.7)]{MM07}.
Thus
\[
\exp(-uD_p^2)(x,x) \in
\bigoplus_j \End(\Lambda^j (T^{*(0,1)}M) \otimes E)_x\,.
\]
We will denote by $\det$ the determinant on $T^{(1,0)}M$.
The following gives the scaling asymptotics for
the Kodaira-Laplacian heat kernel.

\begin{thm}\label{ASYM}
Assume that $M$ is compact.
For $T>0$,  and any $k,m\in \N$
    we have as $p\to\infty$
\begin{equation}\label{lm4.11}
\exp\!\del{\!-\frac{u}{p} D_p^2}\!(x,x)
= \sum_{r=0}^{m}
\del{\frac{p}{u}}^{\!n-r}\! e_{\infty\, r}(u,x)
+ \del{\frac{p}{u}}^{\!n-m-1}\! R_{m+1}\del{\frac{u}{p}, u,x}
\end{equation}
 uniformly for $0<u<T$ and $x\in M$, in the $\cC^k$-norm on
$\cC^\infty(M,\End(\Lambda (T^{*(0,1)}M) \otimes E))$,
i.\,e., the reminder term $R_{m+1}(\frac{u}{p}, u,x)$
is uniformly bounded for $0<u<T$, $x\in M$, $p\in \N^*$.
For any $r\in\N$, the coefficient $e_{\infty\, r}(u,x)$ is smooth
at $u=0$ and the principal term is given by
\begin{align}\label{eq:z.3}
e_{\infty 0}(u,x)
= \frac{1}{(2\pi)^{n}}
\frac{\det (u \dot{R}^L_{x})\exp(2u \omega_{d,x})}
{\det (1-\exp(-2u\dot{R}^L_{x}))} \otimes \Id_{E}.
\end{align}
\end{thm}
The leading term of the scaling asymptotics has been known for
some time in connection with the Demailly
holomorphic Morse inequalities \cite{De:85}. Bismut \cite{B87c}
and Demailly \cite{De:91} used the heat kernel to prove these
inequalities, based on the principal term
of the scaling asymptotics above.  The new feature of
Theorem \ref{ASYM} is the complete asymptotic expansion in the
$\cC^{\infty}$ sense, and the computability of the coefficients.
It is a kind of generalization, in terms of both statement and proof,
of the Bergman/Szeg\"o kernel expansion on the diagonal
given in \cite{Catlin99}, \cite{Z} in the case of positive
Hermitian holomorphic line bundles. The main
feature of the heat kernel expansion is its generality: it does not
require that $(L, h^L)$  be a positive line bundle,
nor even that $(M, \Theta)$ be a K\"ahler manifold. In the general
case, the Bergman/Szeg\"o kernel is difficult  to analyze
and the heat kernel is a good substitute.
Note that for $u>0$ fixed, Theorem \ref{ASYM} was obtained
in \cite[(1.4)]{DLM06}, \cite[(4.2.4)]{MM07}.

Let us give another form of the principal term \eqref{eq:z.3}
in order to recover Demailly's formula \cite[Theorem 4.4]{De:91}.
Let us choose $\{w_j\}_{j=1}^n$ to be an orthonormal basis of
$T^{(1,0)}M$ such that
\begin{align}\label{lm4.4}
\dot{R}^L(x)= {\rm diag}
(\alpha_1 (x), \ldots, \alpha_n(x))\in \End (T^{(1,0)}_{x} M)\,.
\end{align}
The elements $\alpha_1 (x), \ldots, \alpha_n(x)$ are called
the eigenvalues of $R^L$ with respect to $\Theta$.  Then
\begin{equation}  \label{lm4.51}
\om_d(x)=- \sum_{j} \alpha_j(x) \overline{w}^j\wedge
\,i_{\overline{w}_j},  \quad \tau(x)=\sum_j  \alpha_j(x).
\end{equation}
We have by \cite[(1.6.4)]{MM07}
\begin{equation}\label{lm4.52}
e_{\infty 0}(u,x)=u^n\prod_{j=1}^n\, \frac{\alpha_j(x)
\left(1+ (\exp(-2u \alpha_j(x))-1)\ov{w}^j
\wedge i_{\ov{w}_j}\right)}{
2\pi (1- \exp(-2u \alpha_j(x)))}\otimes \Id_E\,.
\end{equation}
Here we use the following convention: if an eigenvalue $\alpha_j(x)$
of $\dot{R}^L_{x}$ is zero, then its contribution to
$\det (\dot{R}^L_{x})/\det(1-\exp(-2u\dot{R}^L_{x}))$
is $1/(2u)$.

Remark that the operator $D_p^2=2\Box_p$ preserves
the $\Z$-grading of
the Dolbeault complex $\Omega^{0,\bullet}(M, L^p\otimes E)$.
We will denote by $\Box^q_p$ the restriction of $\Box_p$
to $\Omega^{0,q}(M,L^p\otimes E)$. We set
\begin{equation}\label{lm4.53}
e^q_p(u,x)=\tr\exp\!\del{-\frac{2u}{p} \Box^q_p}\!(x,x)
=\tr_q \exp\!\del{-\frac{u}{p} D^2_p}\!(x,x)\,.
\end{equation}
where $\tr_q$ is the trace of an operator
acting on $E^q_{p}$. 
By taking the trace
$\tr_q$ of \eqref{eq:z.3} we obtain
\begin{equation}\label{lm4.112}
e^q_p(u,x)
= \sum_{r=0}^{m}
\del{\frac{p}{u}}^{\!n-r}\! e^q_{\infty\, r}(u,x)
+ \del{\frac{p}{u}}^{\!n-m-1}\! R^q_{m+1}\del{\frac{u}{p}, u,x}\,,
\end{equation}
where
\[
e^q_{\infty\, r}(u,x)=\tr_q e_{\infty\, r}(u,x)\,,\quad
R^q_{m+1}\del{\frac{u}{p}, u,x}
=\tr_q R_{m+1}\del{\frac{u}{p}, u,x}\,.
\]
We obtain thus from \eqref{lm4.52},
\begin{equation}  \label{lm4.5}
e^q_{\infty 0}(u,x) = \rank(E)(4\pi)^{-n}\Big(\sum_{|J|= q}
e^{u (\alpha_{\complement J}-\alpha_{J})}\Big)
\prod_{j=1}^n \frac{u\alpha_j(x)}{\sinh (u\alpha_j(x))}\,\cdot
\end{equation}
We use the following notation
for a multi-index $J\subset\{1,\ldots,n\}$:
\[
\alpha_J=\sum_{j \in J} \alpha_j\,,\quad \complement J
=\{1,\ldots,n\}\setminus J\,.
\]
It is understood that
\begin{equation*}
\frac{\alpha}{\sinh \alpha u} = \frac1u\,,\quad
\text{when ${\alpha}=0$}\,.
\end{equation*}

\subsection{Scaling asymptotics of the heat kernel of the
associated principal bundle}

We now state a closely result of the scaling asymptotics of
the heat kernel for the Bochner Laplacian
$\nabla_p^* \nabla_p$.
The method also applies to the Kodaira Laplacian but we
only present it in this case. For simplicity, we do not
twist by a vector bundle $E$.

 As above, we  denote by $(M, \Theta)$ a compact complex
 $n$-manifold with  Hermitian metric $\Theta$,
with volume form $dv_{M} = \frac{\Theta^n}{n!}$, and
let $(L, h^L) \to M$ be a
holomorphic line bundle  with curvature $R^L$.
Let $\nabla_p$ denote the Chern connection
associated to $h_L$ on $L^p$.

Denote  by $L^*$ the dual line bundle and let $D^*_h$ be
the unit disc bundle of $L^*$ with
respect to the dual metric $h^{L^*}$.
The boundary $X = X_h = \partial D^*_h$ is then a principal
$S^1$ bundle $\pi: X \to M$ over $M$.
 The powers $L^p$ of $L$ are the line bundles
 $L^p = X \times_{\chi_p} \C$
associated to the characters
$\chi_p(e^{i \theta}) = e^{i p \theta}$ of $S^1$.
Sections $s $ of $ L^p$ naturally lift to $L^*$ as equivariant
functions $\hat{s}(\lambda) (x) = \lambda( s(\pi(x))$,  and
the lifting map identifies $L^2(M, L^p)$
with the space $L^2_p(X)$ of equivariant functions on $X$
transforming
by $e^{i p \theta}$ under the $S^1$ action on $X$,
which we denote by $e^{i \theta}\cdot x.$
The Chern connection induces an $S^1$-invariant
vertical 1-form $\beta$, defining a connection on $TX$
(see \S \ref{GEOM}).

 We  define the Bochner Laplacian $\Delta^{L^p}$ on $L^p$
by $\Delta^{L^p} = \nabla_p^* \nabla_p$ where $*$ is
taken with respect to $dv_{M}$.  Under the lifiting
identificiation $L^2(M, L^p) \simeq L^2_p(X)$, $\Delta^{L^p}$
corresponds to restriction to $L^2_p(X)$ of  the horizontal
Laplacian $\Delta_H = d_H^* d_H$, where $d_H$ is
the horizontal differential on $X$ for the connection $\beta$.

The lifting identification induces an identifiction of heat kernels,
which takes the following form on the diagonal:
Let   $x \in X, z \in M $ and $\pi(x) = z$. Then
 \begin{equation}\label{FC}  \exp\!\left(\!- \frac{u}{p} \Delta^{L^p}\right) (z,z)
= \int_{S^1}e^{-(u/p) \Delta_H} (e^{i \theta} x, x) %
    e^{- i p \theta} d\theta. \end{equation}
Using this formula, we prove




\begin{thm}\label{ASYMX}
Assume that $M$ is compact. With the above notations and
assumptions,   there exist smooth coefficients
$e^H_{\infty, r}(u, z)$ so that for $T>0$,  and any $k,m\in \N$
    we have as $p\to\infty$

\begin{equation}\label{lm4.11b} \begin{split}
    \exp\!\del{\!-\frac{u}{p} \Delta^{L^p}}\!(z,z)
& = \sum_{r=0}^{m}
\del{\frac{p}{u}}^{\!n-r}\! e^H_{\infty\, r}(u,z)
+ \del{\frac{p}{u}}^{\!n-m-1}\! R_{m+1}\del{\frac{u}{p}, u,z}.
\end{split}
\end{equation}
 uniformly for $0<u<T$ and $z\in M$, in the $\cC^k$-norm on
$\cC^\infty(M)$,
i.\,e., the reminder term $R_{m+1}(\frac{u}{p}, u,z)$
is uniformly bounded for $0<u<T$, $z\in M$, $p\in \N^*$.
\end{thm}
In view of \eqref{FC} we could also state this result as giving
the scaling asymptotics of the $p$th Fourier
component of the horizontal heat kernel.
As  discussed in \S \ref{PARAMETRIX}, the reason for
using \eqref{FC}, is that there exists a rather concrete
Hadamard style parametrix for
$e^{- u \Delta_H}(x, y)$, involving the Hadamard heat kernel
coefficients $\Phi_j$ of a principal bundle,
computed in  \cite[Theorrem 5.8]{BGV}.
All the properties stated in the theorem follow from standard facts
about the
stationary phase method and from the properties of
the coefficients $\Phi_j$.
 The principal term is given by (cf. \eqref{lm4.1} -\eqref{eq:z.3}):
\begin{align}\label{eq:z.3b}
e_{\infty 0}^H(u,x)
= \frac{1}{(2\pi)^{n}}
\frac{\det (u \dot{R}^L_{x}) \exp(-u\tau)}
{\det (1-\exp(-2u\dot{R}^L_{x}))}\,\cdot
\end{align}
Recall that $\tau$ is the trace of the curvature $R^L$ defined in \eqref{lm4.3}.
The subleading term is given by
\begin{equation} \label{SUBPRIN}
\Big(\frac{p}{4 \pi u}\Big)^{\!n} \left[u \Phi_1(x,   2 u ) +
\frac{\partial^2}{\partial \theta^2}
\Phi_0(x, i\theta + 2 u )\big|_{\theta = 0}  \right]\,.
 \end{equation}
Let us compare the expansions \eqref{lm4.11} of the
Kodaira Laplacian and \eqref{lm4.11b} of the Bochner Laplacian.
Note that by Lichnerowicz formula \eqref{alm4.21a}
we have
$D^2_p=\Delta^{L^p}-p\,\tau + \cO(1)$
on $\Omega^{0,0}(M,L^p)$.
Consider the rescaled operator $\widetilde{\cL}^{\,\,0,u}_2$
corresponding to $\frac{u}{p}\Delta^{L^p}$ analogous to
$\cL^{\,\,0,u}_2$ as in \S \ref{MMPF};
the analogue of \eqref{lm4.29} is
$\widetilde{\cL}^{\,\,0,u}_2=-\sum_{i}(\nabla_{0,u,e_{i}})^2$.
Thus the difference between \eqref{eq:z.3b}
and  \eqref{eq:z.3} for $(0,0)$-forms is the factor $\exp(-u\tau)$.

 If one uses the  Lichnerowicz formula to express the
 Kodaira Laplacian in terms of the horizontal (Bochner) Laplacian,
 one may then apply  the  Duhamel formula to
express the heat kernel of the Kodaira Laplacian in terms of that
of the Bochner Laplacian.  Alternatively,
one may go through the parametrix construction as
in \cite[Theorem 5.8]{BGV} but with the Kodaira Laplacian. The
transport equations change  because of the extra curvature term.
We omit the details since we are already giving a proof of
Theorem \ref{ASYM} by another method.
We also leave to the reader the adaptation of the analytic
localization proof of Theorem \ref{ASYM} to obtain
Theorem \ref{ASYMX}.
\subsection{Relation to the holomorphic Morse inequalities}

The original application  of the scaling asymptotics was to
estimating dimensions
 $h^q(L^p\otimes E):= \dim H^q(M, L^p\otimes E) $ of
 holomorphic sections \cite{De:85,De:91,B87c,MM07}.  We follow
the exposition of \cite{MM07,Z2}.

Let $M(q) \subset M$ be the subset in which $\sqrt{-1}R^L$
has precisely $q$ negative
eigenvalues and $n-q$ positive eigenvalues.
Set $M(\leqslant q)=\bigcup_{i=0}^q M(i)$,
$M(\geqslant q)=\bigcup_{i=q}^n M(i)$.

The holomorphic Morse
inequalities of J.-P.\ Demailly \cite{De:85} give asymptotic estimates
for the alternating
sums of the dimensions $h^q(L^p\otimes E)$ as $p\to\infty$.

\begin{thm} \label{lmt5.1}
Let $M$ be a compact complex manifold with $\dim M=n$,
and let $(L,h^L)$, $(E,h^E)$ be Hermitian holomorphic vector
bundles on $M$,
$\rank({L})=1$.
As $p\to\infty$, the following strong Morse inequalities hold
 for every $q=0,1,\dots,n$:
\begin{equation} \label{morse-dem}
\sum\limits_{j=0}^{q}(-1)^{q-j}h^j(L^p\otimes E)
\leqslant \rank({E}) \frac{p^n}{n!} \int_{M(\leqslant q)}(-1)^q
\Big(\tfrac{\sqrt{-1}}{2\pi}R^L\Big)^n+o(p^n)\, ,
\end{equation}
with equality for $q=n$ {\rm(}asymptotic Riemann-Roch-Hirzebruch
formula{\rm)}.

Moreover, we have the weak Morse inequalities
\index{$\rank({E})$ rank of $E$}
\begin{equation} \label{lm5.4}
h^q(L^p\otimes E)\leqslant \rank({E})
 \frac{p^n}{n!}\int_{M(q)}(-1)^q
\left(\tfrac{\sqrt{-1}}{2\pi}R^L\right)^n +o(p^n).
\end{equation}
\end{thm}
%
%
%
%
It was observed by J.-M.~Bismut \cite{B87c}
that the leading order scaling asymptotics of the  heat kernel
could be used to simplify the proof of these
inequalities. Bismut's probability arguments were replaced by
classical heat kernel methods by J.-P.~Demailly \cite{De:91}
and by T.~Bouche
\cite{Bou.1,Bou.2}.  Since one obviously has
\begin{equation} h^q(L^p\otimes E)
    \leq \int_M e^q_p(u,x) dv_{M}  \end{equation}
 for any $u$, we can   let $p \rightarrow \infty$ to obtain
\begin{equation} \limsup_{p \to \infty} p^{-n} h^q(L^p\otimes E)
    \leq \int_M u^{-n} e_{\infty 0}^q(u,x) dv_{M}  \end{equation}
and then let $u \to \infty$ to obtain the
\emph{weak Morse inequalities},
\begin{equation}
    \;\; \limsup_{p \to \infty} p^{-n} h^q(L^p\otimes E)
    \leq \rank({E})  (-1)^q  \frac{1}{n!}
  \int_{M(q)} \Big(\tfrac{\sqrt{-1}}{2\pi}R^L\Big)^n\,.\end{equation}
In the last step we used that
\begin{equation}\label{FACTORS} \mbox{As}\; u \to +\infty,\;\;\;\;\;
  \frac{e^{u \alpha} u \alpha}{\sinh u \alpha}
  = \frac{ 2u \alpha}{1 - e^{-2u \alpha}}  \sim
 \left \{ \begin{array}{ll}   2u \alpha +  O (u e^{- 2 \alpha u})\,,
 & \alpha > 0\,, \\  & \\
1\,, & \alpha = 0\,, \\ & \\
O(u e^{- 2 |\alpha| u}), & \alpha < 0\,. \end{array} \right.
\end{equation}
In the case of $q = 0$ we have:
\begin{equation}  e_{\infty 0}^0(u, x) \sim   \rank({E})\left \{
 \begin{array}{ll}   (4\pi)^{-n}(2u)^{ r}
     \prod_{i, \alpha_i(x)>0} \alpha_i(x),
     & \alpha_i(x) \; \geq 0 \;\; \forall i, \;
     r=\sharp\{i:\alpha_i(x) > 0\}\, ,
     \\  & \\
1 & \alpha_i(x) = 0\,, \;\forall i,\\
& \\0  & \exists i: \alpha_i(x) < 0.  \end{array} \right.
\end{equation}
For general $q$ the asymptotics depends in a more complicated
way on the eigenvalues of $\sqrt{-1}R^L$.  Assume first
that $x \in M(q)$ so $\sqrt{-1}R^L$ is non-degenerate at
$x$ and  let $J_{-}(x)$ denote the  set of $q$ indices for which
$\alpha_j(x) <0,$ resp.
 $J_+(x)$ denote the set of indices for which $\alpha_j(x) > 0.$
 The only term  $\alpha_{\complement J}-\alpha_{J}$ which
makes a non-trivial asymptotic contribution is the one for which
$J = J_-(x).$  Hence
\begin{equation}
\begin{split}
u^{-n} e_{\infty 0}^q(u, x)  & \sim
 \rank({E}) (4\pi)^{-n}u^{- n }   \prod_{j} e^{u |\alpha_j(x)|}
 \frac{u \alpha_j(x)}{\sinh (u \alpha_j(x))}\\ & \\
& \sim \rank({E}) (-1)^q \prod_{j}\dfrac{\alpha_j(x)}{2\pi}\,\cdot
\end{split}
\end{equation}
Now assume that the curvature is degenerate at $x$, with
$n_-$ negative eigevalues, $n_0$ zero eigenvalues and
$n_+$ positive eigenvalues.  Since we must change the sign
of $q$ eigenvalues and since any negative eigenvalue causes
the whole product to vanish in the $u \to \infty$ limit,
the asymptotics are trivial unless $n_- \leq q$ and
$n_- + n_0 \geq q.$
If $n_- = q$ there is only one term in the $J$ sum, namely
where $J = J_-(x).$ If $n_- < q$ then we may choose any
$q - n_-$ indices of zero eigenvalues to flip. There are
$\binom{n_0}{q-n_-}$such indices. Hence in the degenerate case
with $n_- \leq q$, $n_- + n_0 \geq q$ we have
\begin{equation}\begin{split} u^{-n} e_{\infty 0}^q(u, x)  & \sim
 \rank(E)u^{- n} {\binom{n_0}{q-n_-}}  (4\pi)^{-n}
 \prod_{j \in J_+(x) \cup J_-(x)} e^{u |\alpha_j(x)|}
 \frac{u \alpha_j(x)}{\sinh u \alpha_j(x)}\\ & \\
& \sim  \rank(E)u^{- n_0 } {\binom{n_0}{q-n_-}}
(4\pi)^{-n}(-1)^{n_{-}}
\prod_{j \in J_+(x) \cup J_-(x)}  (2\alpha_j(x)). \end{split}
\end{equation}
Thus, one first takes  the limit $p \to \infty$ and then
$u \to \infty$. A natural question is whether
one can let $u \to \infty, p \to \infty$  simultaneously in the
scaling  asymptotics of  Theorem \ref{ASYM}.
 Suppose  that $\sqrt{-1}R^L$ has rank $\leq n - s$ at all points.
 Then one would conjecture that
 $h^q(L^p\otimes E) \leq \varepsilon(p)^s p^{n - s} $ as
 $p \to \infty$ where $\varepsilon(p)$ is any function such that
 $\varepsilon(p) \uparrow \infty$ as $p \to \infty.$
 Let $u(p) = \frac{p}{\varepsilon(p)}.$  Suppose that
 $u(p)/p = 1/\varepsilon(p)$ could be used as a small parameter
 in the expansion
of Theorem \ref{ASYM}. The principal term is of order
$(\frac{p}{u(p)})^n u(p)^{n - s} = \varepsilon(p)^{s} p^{n - s}$
and one would hope that the remainder is of order
$\varepsilon(p)^{s+ 1} p^{n - 1 - s}.$
Our remainder estimate in Theorem \ref{ASYM} is not
sharp enough for this application.

Let us close by reminding the proof of
the \emph{strong Morse inequalities} \eqref{morse-dem}
(cf.\ \cite[\S1.7]{MM07}).
As before, we denote by $\tr_q \exp\!\del{-\frac{u}{p} D^2_p}$
the trace of
$\exp\!\del{-\frac{u}{p} D^2_p}$ acting on
$\Omega^{0,q}(M, L^p\otimes E)$.
Then we have (using notation \eqref{lm4.53})
\begin{align}\label{lm5.6}
\tr_q \exp\!\del{-\frac{u}{p} D^2_p}\!
=\int_M e^q_p(u,x)dv_{M} (x).
\end{align}
By a linear algebra argument involving the spectral spaces
\cite[Lemma 1.7.2]{MM07} we have for any $u>0$ and
$q\in \N$ with $0\leqslant q \leqslant n$,
\begin{align}\label{lm5.7}
\sum_{j=0}^q (-1)^{q-j} h^j(L^p\otimes E)
\leqslant \sum_{j=0}^q (-1)^{q-j}
\tr_j \exp\!\del{-\frac{u}{p} D^2_p},
\end{align}
with equality for $q=n$.
Note that in the notation of \eqref{lm4.3},
\begin{align}\label{lm4.77}
\exp\!\big(\!-2u \alpha_j(x_0)\ov{w}^j\wedge i_{\ov{w}_j}\big)
=1+ \big(\exp(-2u \alpha_j(x_0))-1\big) \ov{w}^j
\wedge i_{\ov{w}_j}.
\end{align}
We denote by $\tr_{\Lambda ^{0,q}}$ the trace on
$\Lambda ^q (T^{*(0,1)}M)$. By \eqref{lm4.77},
\begin{align}\label{lm5.11}
\tr_{\Lambda ^{0,q}} [\exp(2 u \om_d)]
= \sum_{j_1<j_2<\cdots <j_q} \exp\!\Big(\!- 2 u
\sum_{i=1}^q \alpha_{j_i}(x)\Big).
\end{align}
Thus by \eqref{lm4.11},
\[
\frac{\det (\dot{R}^L/(2\pi))}
{\det(1- \exp(-2 u\dot{R}^L))}
 \tr_{\Lambda ^{0,q}} \!\big[\exp(2 u \om_d)\big]
 \]
is uniformly bounded in $x\in M$, $u>1$,
$0\leqslant q \leqslant n$. Hence
for any $x_0\in M$, $0\leqslant q \leqslant n$,
\begin{align}\label{lm5.12}
\lim_{u\to \infty} \frac{\det (\dot{R}^L/(2\pi))
\tr_{\Lambda ^{0,q}}\!\big[\exp(2 u \om_d)\big]}
{\det(1- \exp(-2 u\dot{R}^L))} (x_0)
= {\mathds 1}_{ M(q)}(x_0) (-1)^q
\det \Big(\frac{\dot{R}^L}{2\pi}\Big) (x_0)\,,
\end{align}
where ${\mathds 1}_{ M(q)}$ is the characteristic function
of $M(q)$.
%
{}From  Theorem \ref{ASYM}, \eqref{lm5.6} and \eqref{lm5.7},
we have
\begin{align}\label{lm5.13}
\begin{split}
&\limsup_{p\to\infty} p^{-n}\sum_{j=0}^q (-1)^{q-j}
h^j(L^p\otimes E) \\
&\hspace*{5mm}\leqslant \rank (E)
\int_M \frac{\det (\dot{R}^L/(2\pi))
\sum_{j=0}^q (-1)^{q-j}\tr_{\Lambda ^{0,j}}\!\big[\exp(2 u \om_d)\big]}
{\det(1- \exp(-2 u\dot{R}^L))}\,dv_{M} (x)\,,
\end{split}
\end{align}
for any $q$ with $0\leqslant q \leqslant n$ and any $u>0$.
%
Using \eqref{lm5.12}, \eqref{lm5.13} and dominate convergence,
we get
\begin{align}\label{lm5.14}
\limsup_{p\to\infty} p^{-n}\sum_{j=0}^q (-1)^{q-j}
h^j(L^p\otimes E)
\leqslant (-1)^q \rank (E) \int_{\bigcup_{i=0}^q M(i)}
\det \Big(\frac{\dot{R}^L}{2\pi}\Big)(x)\,dv_{M} (x).
\end{align}
But \eqref{lm4.4} entails
\begin{align}\label{lm5.15}
\det \Big(\frac{\dot{R}^L}{2\pi}\Big)(x) dv_{M} (x)
=\prod_j \frac{a_j(x)}{2\pi}dv_{M} (x)
= \frac{1}{n!}\Big(\frac{\sqrt{-1}}{2\pi}R^L\Big)^{\!n}.
\end{align}
{}Relations \eqref{lm5.14}, \eqref{lm5.15} imply \eqref{morse-dem}.

Let us finally mention that the original proof of Demailly of the
holomorphic Morse inequalities
is based on the asymptotics of the spectral function of the
Kodaira Laplacian $\Box^q_p$
as $p\to\infty$, the semiclassical Weyl law,
cf.\ \cite{De:85}, \cite[Theorem\,3.2.9]{MM07};
see also \cite{HsM14} for the lower terms of the asymptotics
of the spectral function.

\subsection{Organization and notation}

We end the introduction with some remarks on the organization
and notation.

We first prove Theorem \ref{ASYMX}. The proof involves the
construction of a Hadamard parametrix for the heat kernel of
a principal $S^1$ bundle adapted from \cite{BGV},
and in particular involves the connection, distance function
and volume form on the $S^1$ bundle $X_h \to M$.
The geometric background is presented in \S \ref{S1}, and then
the proof of Theorem \ref{ASYMX} is  given in \S \ref{PFX}.
The main complication is that
the heat kernel must be analytically continued to $L^*$, which is
an important but relatively unexplored aspect of heat kernels.
The proof of Theorem \ref{ASYM} is then given in  \S \ref{MMPF}.

We now record some basic definitions and notations.

\subsubsection{The complex Laplacian}

Let $F$ be a holomorphic Hermitian vector bundle over
a complex manifold $M$. Let $\Omega^{r,q}(M,F)$
be the space of smooth  $(r,q)$-forms on $X$ with values in $F$.
Since $F$ is holomorphic
the operator $\db^F:\cC^\infty(M,F)\to\Omega^{0,1}(M,F)$
is well defined.
A connection $\nabla^F$ on $F$ is said to be a
holomorphic connection
if $\nabla^F_Us=i_U(\db^Fs)$ for any $U\in T^{(0,1)}M$
and $s\in\cC^\infty(X,F)$.
It is well-known that there exists a unique holomorphic Hermitian
connection
$\nabla^E$ on $(F,h^F)$, called the {\em Chern connection}.

The operator $\db^F$ extends naturally to
$\db^F:\Omega^{{\bullet},{\bullet}}(M,F)\longrightarrow
\Omega^{{\bullet},{\bullet\,+1}}(M,F)$ and $(\db^F)^2=0$.
Let $\nabla^F$ be the Chern connection
on $(F,h^F)$. Then
we have a decomposition of $\nabla^F$ after bidegree
\begin{equation}\label{herm13}
\begin{split}
&\nabla^F=(\nabla^F)^{1,0}+(\nabla^F)^{0,1}\,,
\quad (\nabla^F)^{0,1}=\db^F\,,\\
&(\nabla^F)^{1,0}:\Omega^{{\bullet\,,\,\bullet}}(M,F)
\longrightarrow \Omega^{{\bullet}\,+1\,,\,{\bullet}}(M,F)\,.\\
\end{split}
\end{equation}
The Kodaira Laplacian is defined by:
\begin{equation}\label{herm17}
{\square}^F=\big[\,\db^F,\db^{F,*}\,\big]\,.
\end{equation}
Recall that the Bochner Laplacian $\Delta ^F$ associated to
a connection $\nabla^F$ on a bundle $F$ is defined
(in terms of a local orthonormal frame $\{e_i\}$ of $TM$) by
\begin{equation}\label{lm1.21}
\Delta^F:=-\sum_{i=1}^{2n} \left ((\nabla^F_{e_i})^2
-\nabla^F_{ \nabla ^{TM}_{e_i}e_i}\right )\,,
\end{equation}
where $\nabla ^{TM}$ is the Levi-Civita connection on
$(TM,g^{TM})$.  Moreover, $\Delta^F=(\nabla^F)^*\,\nabla^F$,
where the adjoint of $\nabla^F$ is taken with respect to
$dv_M$ (cf.\ \cite[(1.3.19), (1.3.20)]{MM07}) .

\subsubsection{\label{NOTATION} Notational appendix}

\begin{itemize}

\item $(M, J)$ is a complex manifold with complex structure $J$.

\item $g^{TM}$ is a $J$-invariant Riemannian metric on $M$,
$\Theta (\cdot, \cdot):= g^{TM}(J\cdot, \cdot)$.

\item $dv_{M}$ is the Riemannian volume form   of $(TM, g^{TM})$.

\item $(L, h^L, \nabla^L)$ is a holomorphic  line bundle
with Hermitian metric $h^{L}$
and Chern connection $\nabla^{L}$;  $(E, h^{E}, \nabla^{E})$ is a
holomophic vector bundle $E$ with Hermitian metric $h^{E}$
and Chern connection $\nabla^{E}$ on $M$.

\item $\omega : = \frac{\sqrt{-1}}{2\pi}R^L$;
$\dot{R}^L \in \End(T^{(1,0)}M)$ is the associated Hermitian
endomorphism \eqref{lm4.2};
the derivation $\om_d$ and the trace $\tau(x)$ of $R^L$ are defined in \eqref{lm4.3}.

\item Dolbeault-Dirac operator
$D_p = \sqrt{2}\Big(\overline{\partial}^{L^p\otimes E}
+ \,\overline{\partial}^{L^p\otimes E,*}\Big)$.

\item Kodaira Laplacian
$\square_p=\tfrac12 D^2_p= \overline{\partial}^{L^p\otimes E}\,
\overline{\partial}^{L^p\otimes E,*}
+\,\overline{\partial}^{L^p\otimes E,*}\,
\overline{\partial}^{L^p\otimes E}\,.$

\item $\Box^q_p$ is the resetriction of $\Box_p$
to $\Omega^{0,q}(M,L^p\otimes E)$.

\item $\nabla_p$ is the Chern connection of $(L^p, h^{L^p})$.
The Bochner Laplacian  $\Delta^{L^p}$ on $L^p$
is  $\Delta^{L^p} = \nabla_p^* \nabla_p$ where $*$ is
taken with respect to $dv_{M}$.

\item $X_h = X = \partial D^*_h $ where $D^*_h \subset L^*$
is the unit co-disc bundle. Then   $\Delta^{L^p}$
can be identified with the horizontal Laplacian $\Delta_H$ on $X$.
\end{itemize}

\section{\label{S1}Heat kernels on the principal $S^1$ bundle}

In this section, we prepare for the proof of Theorem \ref{ASYMX}
by reviewing the geometry of the principal $S^1$ bundle
$X_h \to M$ associated to a Hermitian holomorphic line bundle.
The geometry is discussed in more detail in \cite{B.G,  BSj, Z}
but mainly under the assumption
that $(L, h^L)$ is a positive Hermitian line bundle.
We do not make this assumption in Theorem \ref{ASYMX}.

  As above, we denote by  $(M, \Theta)$ a compact complex
  $n$-manifold with a Hermitian metric $\Theta$ and
then $dv_{M} = \frac{\Theta^n}{n!}$ is  its Riemannian volume form.
We consider a Hermitian
holomorphic line bundle $(L,h^L) \to M$ with curvature $R^L$.
Denote the eigenvalues of 
$\dot{R}^L$ relative to $\Theta$ by
$\alpha_1(x), \dots, \alpha_n(x)$ (cf. \eqref{lm4.4}).

Denote  by $L^*$ the dual line bundle and let $D_h^*$ be
the unit disc bundle of $L^*$ with
respect to the dual metric $h^{L^*}$.
The boundary $X = \partial D_h^*$ is then a principal
$S^1$ bundle $\pi: X \to M$ over $M$, and we denote the $S^1$
action by $e^{i \theta}\cdot x.$
 We may express the powers $L^p$ of $L$ as
 $L^p = X \times_{\chi_p} \C$
where $\chi_p(e^{i \theta}) = e^{i p \theta}.$
Sections $L^2(M, L^p)$ of $L^p$ can be naturally identified
with the space $L^2_p(X)$ equivariant functions on $X$
transforming
by $e^{i p \theta}$ under the $S^1$ action on $X$.

\begin{rem} The condition  that $(L, h^L)$ be a positive line bundle
is equivalent to the condition that  $D_h^*$ be a strictly
pseudo-convex domain in $L^*$, or equivalently that $X_h$
be a strictly pseudo-convex CR manifold. These assumptions are
used in \cite{BSj} to construct parametrices for the Szeg\"o kernel,
but are not necessary to construct  parametrices for the heat
kernels.
\end{rem}

\subsection{\label{GEOM} Geometry on a circle bundle}

The Hermitian metric $h^L$ on $L$ induces a connection
$\beta = (h^L)^{-1} \partial h^L$ on the $S^1$ bundle $X \to M$,
which is invariant under the $S^1$ action and satisfies
$\beta ( \frac{\partial}{\partial \theta}) = 1$.
Here, $\frac{\partial}{\partial \theta}$ denotes the generator
of the
action.  We denote by $V_x=\R\frac{\partial}{\partial \theta} $
its span in $T_xX$ and by $H_x= \ker \beta$ its horizontal
complement. We further equip $X$ with
the \emph{Kaluza-Klein bundle metric} $G$, defined by declaring
the vertical and horizontal
spaces orthogonal, by equipping the vertical space with the
$S^1$ invariant metric and by equipping
the horizontal space with the lift of the metric $g^{TM}$.
More precisely, it is defined by
the conditions:
$H \bot V$, $\pi_*: (H_x, G_x) \to (T_{\pi(x)}M, g^{TM}_{\pi(x)})$
is an isometry and $|\frac{\partial}{\partial \theta}|_G = 1$.
The volume form of the Kaluza-Klein metric on $X$ will be denoted
by $dv$.

The connection $\beta$ on $TX$ induces by duality a connection
on the cotangent bundle $T^*X$, defined by  $ H^*_x: = V^o_x,$
and $ V^*_x = H^o_x$.  Here,
$F^o$ denotes the annihilator of a subspace $F \subset E$
in the dual space $E^*.$ By definition we have
$V^*_x = \R \beta_x$.  The vertical, resp.\ horizontal components
of $\nu \in T^*_x X$  are given by:
\begin{equation}   \nu_V:= \Big\langle \nu,
    \frac{\partial}{\partial \theta} \Big\rangle \beta_x,\;\;\;
\nu_H = \nu - \nu_V. \end{equation}
There also exists   a natural pullback map $\pi^*: T^*M \to T^*X.$
It is obvious that $\pi^* T^*_xM \subset V^o_x$
and as the two sides have the same dimension we see that
$\pi^* T^*_xM = H^*_x X.$ We  write $\pi_*$ with
a slight abuse of notation for the inverse map
$\pi_*: H^*_x X \to T^*_xM.$

This duality can also be defined by the metric $G$, which
induces isomorphisms $\widetilde{G}_x: T_x X \to T^*_x X,$
$\widetilde{G}(X) = G(X, \cdot).$ We note that by definition of $G$,
$\beta_x = G(\frac{\partial}{\partial \theta}\,, \cdot)$
hence $\widetilde{G}: V_x \to V_x^*.$ Similarly,
$\widetilde{G}_x : H_x \to H_x^*.$

\subsection{Distance function on $X$}
To construct a parametrix for the heat kernel, we will need a
formula for the square $d^2(x,y)$  of
the Kaluza-Klein geodesic distance function  on a neighborhood
of the diagonal of $X \times X$.

We first describe the geodesic flow and exponential map
of $(X, G)$.
It  is convenient to identify $TX \equiv T^*X$ as above,
and to consider the (co-)geodesic flow on $T^*X$. This is
the Hamiltonian flow of the metric function
$|\xi|^2_G = |\xi_H|_G^2 + |\xi_V|^2.$  We note that
$|\xi_V|^2 = \langle \xi,   \frac{\partial}{\partial \theta} \rangle^2$
and that $|\xi_H|_G^2 = |\pi_* \xi_H|_g^2.$
Henceforth we put $p_{\theta}(x, \xi):=
\langle \xi,   \frac{\partial}{\partial \theta} \rangle$.
The Hamiltonian flow of $p_{\theta}$ is the
lift to $T^*X$ of the $S^1$ action on $X$, i.\,e.,
$V^u(x, \xi) = (e^{i u} x, e^{i u} \xi)$
where $e^{it} \xi$ denotes derived action of $S^1$ on $X$.
We also denote the Hamiltonian flow of $|\xi_H|_G^2$ by
$G^t_H(x, \xi).$

Let $U \subset M$ denote a trivializing open set for $X \to M$,
and let $\mu: U \to X$ denote a local unitary frame.  Also,
let $z_i, \bar{z}_i$ denote local coordinates in $U$.
Together with $\mu$ they induced local coordinates
$(z_i, \bar{z}_i, \theta)$ on $\pi^{-1}(U) \sim U \times S^1$
defined by $x = e^{i \theta} \cdot \mu(z, \bar{z}).$
Thus, $z_i$ is the pull-back $\pi^* z_i$ while $\theta$ depends
on a slice of the $S^1$ action.
They induce local coordinates $(z, p_z, \theta, p_{\theta})$
on $T^*( \pi^{-1}(U)) \subset T^*X$ by
the  rule $\eta_x = p_z dz + p_{\theta} d \theta.$
Here we simplify the notation $(z,\bar{z})$ for local coordinates
to $z$ until we need to emphasize the complex structure.
As a pullback,  $dz \in H^*_z$ and the notation $p_z dz$
for a form $\eta$ on $M$ and its pullback to $X$ are compatible.
 We also note that the canonical symplectic form
$\sigma$ on $T^*X$ is given by
$\sigma = dz \wedge d p_z + d \theta \wedge d p_{\theta}.$
Hence the Hamiltonian
vector field $\Xi_{p_{\theta}}$ of $p_{\theta}$ is given by
$\Xi_{p_{\theta}}= \frac{\partial}{\partial \theta}.$

It follows that the  Poisson bracket
$\{ |\xi_H|_g, p_{\theta}\} = \frac{\partial}{\partial \theta}
|\xi_H|_g = 0$, so $p_{\theta}$ is a constant
of the geodesic motion and
$\widetilde{G}^u = \widetilde{G}_H^u \circ V^u
= V^u \circ \widetilde{G}_H^u.$
In local coordinates
the Hamiltonian $|\xi|_G^2$ has the form
\begin{equation}
|\xi|_G^2 = \sum_{i,j = 1}^n g^{i j}(z) p_z^i p_z^j + p_{\theta}^2
\end{equation}
and the equations of the Hamiltonian flow of $|\xi|_G^2$ are:
\begin{equation} \left\{ \begin{array}{l} \dot{z}_i
    =  \sum_{j = 1}^n g^{i j}(z)  p_z^j \\ \\
\dot{p_z}_k = \frac{\partial}{\partial z_k}
\sum_{i,j = 1}^n g^{i j}(z) p_z^i p_z^j \\ \\
\dot{\theta} = 2 p_{\theta}\\ \\
\dot{p_{\theta}} = 0. \end{array} \right. \end{equation}
These equations decouple and we see that
\begin{equation}
 G^u(z, p_z, \theta, p_{\theta})
 = (G^t_M(z, p_z), \theta + 2 u p_{\theta}, p_{\theta}).
\end{equation}
It follows that the (co-) exponential map $\exp_x: T^*_xX \to X$
is given at $x = (z, \theta)$  by
\begin{equation}
    \exp_x(p_z, p_{\theta}) = e^{i 2 p_{\theta}} \exp_x (p_z, 0)
    = e^{i 2 p_{\theta}} \mu ( \exp_{z M} p_z),
\end{equation}
where $\exp_{zM}$ is the exponential map of $(M,g^{TM}).$
We observe that $G^u(x, \xi_H) = G^t_H(x, \xi_H)$.

Let $d_M(x, y)$ be the distance on $M$ from $\pi(x)$ to $\pi(y)$.
We  claim:

\begin{lemma} \label{Dist}  In a neighborhood of the diagonal
    in $X \times X$, the distance function satisfies the
Pythagorean identity
\begin{equation}
d(x, y)^2 = d_M(\pi(x), \pi(y))^2 + \theta_0^2
\end{equation}
where we write $y = e^{i \theta_0} \gamma_H(L)$ if
$\gamma_H(u)$ is the horizontal lift to $x$ of the geodesic
$\gamma(u)$ of $(M,g)$ with
$\gamma(0) = \pi(x), \gamma(L) = \pi(y).$  \end{lemma}

Indeed, by definition $d(x, y)^2 =
|\xi|_G^2 = |\xi_H|^2 + p_{\theta}^2$ where $\exp_x \xi = y.$
Now, $\gamma_H(u) $ must be  the first component of
$G^t_H(x, \xi_H)$ since both curves are horizontal lifts of
the same geodesic on $M$ with the same initial condition. Hence
$|\xi_H|^2 = d_M(\pi(x), \pi(y))^2.$
Further, $\theta_0 = p_{\theta}(\xi)$ by Hamilton's equations.

\subsection{Volume density of the Kaluza-Klein metric}

The heat kernel parametrix also involves the volume form of $(X, G)$.
Hence we consider Jacobi fields and  volume distortion under
the geodesic flow.
Following \cite[Section 5.1]{BGV}, we define
\begin{equation} \label{J} J(x, a): T_x X \to T_{\exp_x a}  X
\end{equation}
as the derivative $d_a \exp_x$ of the exponential map of
the Kaluza-Klein metric
at $a \in T_x X.$  We identify $H_x \equiv T_xM$ and
$V_x \equiv \R$.
  For completeness we sketch the proof.
\begin{prop}[{\cite[Theorem 5.4]{BGV}}]\label{JDEF}  Let
    $a = a \frac{\partial}{\partial \theta}$.
    The map $J(x,a)$ preserves the subspaces $H_x$ and $V_x.$
    Moreover
$$\begin{array}{l}
J(x,a)|_{H_x} = \dfrac{1 - e^{- \tau(\omega_x \cdot a)/2}}
{\tau(\omega_x \cdot a)/2},\\ \\
J(x, a)|_{V_x} = \Id. \end{array}$$ \end{prop}

Here, $\tau$ is defined as in the proof of
\cite[Proposition 5.1]{BGV}. Namely, for a vector space $V$,
$ \tau: \bigwedge^2 V \simeq so(V)$ is defined by
\begin{equation} \label{taudef}
    \langle \tau(\alpha) e_i, e_j \rangle
    = 2 \langle \alpha, e_i \wedge e_j \rangle. \end{equation}

\begin{proof} We need to compute
    $\frac{d}{dt}\big|_{u = 0} \exp_x(a + u\xi_H).$  Let
$Y(s) = \frac{d}{dt}\big|_{u = 0} \exp_x s (a + u \xi_H).$
Then $Y(s)$ is a Jacobi field along $\exp_x (s a)$
and $J(x,a) \xi_H = Y(1).$ Hence
$$\frac{D^2}{ds^2} Y + R(T, Y) T = 0$$
where $T = \frac{\partial}{\partial \theta} $ is the tangent vector
to $\exp_x (s a).$ It is easy to see that
$Y(s)$ is horizontal, hence that
$$R(T, Y) T = \frac{1}{4} \tau(\omega_x \cdot a)^2 Y.$$
Using a parallel horizontal frame along $\exp_x (sa)$ we identify
$Y(s)$ with a curve in $H_x$ and  find that
$$\Big(\frac{\partial}{\partial s}
+ \tau(\omega \cdot a)/4\Big)^2 y(s)
= \big(\tau(\omega \cdot a)/4\big)^2 y(s).$$
The formula follows.
\end{proof}
We then have
\begin{cor}[{\cite[Corollary 5.5]{BGV}}] We have
\[\det J(x, a) = j_H( \tau(\omega_x \cdot a)/2)\,,
\:\:\text{where $j_H(A) = \det\Big(\frac{\sinh(A/2)}{A/2}\Big)$}.
\]
\end{cor}


\section{\label{PFX} Proof of Theorem \ref{ASYMX} }

In this section we prove Theorem \ref{ASYMX} for the heat kernels
associated to the  Bochner Laplacians
by $\Delta^{L^p} = \nabla_p^* \nabla_p$ on $L^p$  where $*$ is
taken with respect to $dv_{M}$.
 As discussed in the introduction,
we use \eqref{FC} to  analyze  heat kernels on $L^p$ by
identifying them as Fourier coefficients of heat kernels
on $X$.

 There is a natural lift of sections of $L^p$ to equivariant functions
 on $X$.  Under the
identification $L^2(M, L^p) \simeq L^2_p(X)$, $\Delta^{L^p}$
corresponds to the horizontal
Laplacian $\Delta_H = d_H^* d_H$, where $d_H$ is
the horizontal differential on $X$. We then add the
vertical Laplacian to define the (Kaluza-Klein) metric Laplacian
$\Delta^X = (\frac{\partial}{\partial \theta})^2 + \Delta_H$
on $X$.

Since
 $[(\frac{\partial}{\partial \theta})^2, \Delta_H] = 0$ we have
\begin{equation} e^{- u \Delta^{L^p}}
    = e^{u (\frac{\partial}{\partial \theta})^2}
e^{- u \Delta^X}. \end{equation}
This shows that the equation \eqref{FC} is correct, i.\,e.,
if $\pi(x) = z$,
\begin{equation}\label{FCb}
    \exp\big(-\frac{u}{p} \Delta^{L^p}\big)(z,z)
    = e^{u p} \int_{S^1}e^{-(u/p) \Delta^X} (e^{i \theta} x, x)
    e^{- i p \theta} d\theta.
\end{equation}

\begin{rem} The purpose of adding the vertical term
$(\frac{\partial}{\partial \theta})^2$ is
that there exists a simple parametrix for $e^{- u \Delta^X}$.
Without adding the vertical term, the horizontal heat kernel is
much more complicated and reflects
the degeneracies of the horizontal
curvature. The model case of
the Heisenberg sub-Laplacian only applies directly when
the line bundle is positive.

The Fourier formula \eqref{FCb} shows that from a spectral point
of view, the addition of  $(\frac{\partial}{\partial \theta})^2$
is harmless. But from a Brownian motion point of view it is drastic
and it is responsible for the necessity of analytically
continuing the heat kernel to $L^*$.  For further remarks see
\S \ref{REMSECT}.
\end{rem}

\subsection{\label{PARAMETRIX} Parametrices for heat kernels
}

We will use the   construction in \cite{BGV} of a parametrix for
the heat kernel on a principal $S^1$
bundle $\pi: X \to M$. For remainder estimates, we use
the off-diagonal estimates of Kannai \cite{K}, which
apply to these and more general heat kernel constructions.

To get oriented, let us review the general
Hadamard (-Minakshisundararam-Pleijel) parametrix construction
on a general Riemannian manifold.
Let $(Y,g)$ denote a complete Riemannian manifold of
dimension $m$ and let $\Delta_g$ denote its (positive) Laplacian.
For $x$ close enough to $y$,  the heat kernel of
$e^{- u \Delta_g }$ has an asymptotic expansion as
$u \to 0$ of the form
\begin{equation}
    e^{-u \Delta_g }(x, y) \sim \frac{1}{(4 \pi u)^{m/2}}\,
e^{-d(x,y)^2/4u} \sum_{j = 0}^{\infty} v_j(x, y) u^j \end{equation}
where $d(x,y)$ denotes the distance between $x$ and $y$ and
where $v_j$ are coefficients satisfying certain well-known
transport equations. More precisely, choose a cut-off function
$\psi(d(x,y)^2)$, equal one in a neighborhood $U$ of the
diagonal.  Then there exist smooth locally defined functions
$v_j(x,y)$ such that
\begin{equation} H_M(u, x, y):=
    \psi(d(x,y)^2) \frac{1}{(4 \pi u)^{m/2}}\, e^{-d(x,y)^2/4u}
    \sum_{j = 0}^{M}
v_j(x, y) u^j, \end{equation}
is a parametrix for $ e^{-u \Delta_g }(x, y)$, i.\,e.,  in $U$
we have
\begin{equation} \label{para}
\Big(\frac{\partial}{\partial u} + \Delta_g\Big) H_M(u, x, y) =
\psi(d(x,y)^2) \frac{1}{(4 \pi u)^{m/2}} \,e^{-d(x,y)^2/4u}
\Delta_g v_M(x, y) u^M.\end{equation}
It follows by the off-diagonal estimates of
Kannai (\cite[(3.11)]{K}) that
\begin{equation}\label{NEAR}
    e^{- u \Delta_g}(x, y) - H_M(u, x, y) = O(u^{M - m /2})\,
    e^{- d(x, y)^2/4u}  \end{equation}
for $(u, x, y)$ satisfying: $u < \operatorname{inj}(x)$
and $y \in B(x,\operatorname{inj}(x))$.
Here, $\operatorname{inj}(x)$ is the injectivity
radius at $x$ and $B(x,r)$ is the geodesic ball of radius $r$
centered at $x$.

\subsection{Heat kernels on $S^1$ bundles and $\R$ bundles}
We apply the heat kernel parametrix construction to the principal
$S^1$ bundle
$X = X_h \to M$ equipped with the Kaluza-Klein metric.
As a special case where
$M = \mbox{pt}$, this gives a heat kernel parametrix on the circle
$S^1$. It is standard to
express the heat kernel on $S^1$ as the projection over $\Z$
of the heat kernel on $\R$, since
the distance squared is globally defined on $\R$ but not on $S^1$.
It  thus simplies
 the analysis  of the distance function to consider the principal
$\R$ bundle $\tilde{\pi}: \widetilde{X} \to M$, where
$\widetilde{X}$ is the fiberwise universal cover of $X$.
We thus express the heat kernel
on $X$ as the projection to $X$ of the heat kernel
$e^{-u \Delta^{\widetilde{X}}}(\widetilde{x}+ i\theta, \widetilde{x})$ on
$\widetilde{X}$, where $\widetilde{X}$ is equipped with the
Kaluza-Klein metric, so that
$p: \widetilde{X} \to X$ is a Riemannian $\Z$-cover.
We use additive notation for the $\R$ action on the  fiber
of $\tilde{X}$, i.\,e., in place of the $S^1$ action
$e^{i \theta}  \cdot x$ on $X$ we write
$\widetilde{x} + i \theta$ on $\widetilde{X}.$

A key point in the formula \eqref{FCb} is that in
Theorem \ref{ASYMX} we only
use the heat kernel at points $(x, y)$ where $\pi(x) = \pi(y)$.
The same is true
when we lift to $\widetilde{X}$. Although we need to construct
a heat kernel parametrix off the diagonal, it is only evaluated at
such off-diagonal points. Hence
it is sufficient to use a base cut-off, of the form
$\psi(d_M(x, y)^2)$ where as above,
$d_M(x, y)$ is the distance on $M$ from $\pi(x)$ to $\pi(y)$.
This cutoff is identically equal to one on points
$(x,y)$ on the same fiber.

 The  following proposition is adapted from
 \cite[ Theorems 2.30]{BGV}.

\begin{prop} \label{RLIFT} There exist smooth functions
    $\widetilde{\Phi}_\ell$ on
    $\widetilde{X} \times i \R$ such that
$$e^{-u \Delta^{\widetilde{X}}}
(\widetilde{x}, \widetilde{ x}+ i \theta) =  (4 \pi u)^{- (n + 1/2)}
\sum_{\ell = 0}^{M} u^\ell
\widetilde{\Phi}_\ell(\widetilde{x}, i \theta)
\widetilde{j} (\widetilde{x}, i \theta)^{-1/2}
 e^{-|\theta|^2/4u} + R_M(\widetilde{x}, i \theta)  $$
where $\widetilde{j} (\widetilde{x}, i \theta)$ is the volume
density $j(\widetilde{x}, \widetilde{y})$ at
$\widetilde{y} = \widetilde{x} + i \theta$
 in normal coordinates centered at $\widetilde{x}$ and where

$$\begin{array}{l} \widetilde{ \Phi_0}(\widetilde{x}, i \theta) = 1\,,
\ \\
R_M(\widetilde{x}, i \theta) \ll (4 \pi u)^{- (n + 1/2 + M)}
e^{-|\theta|^2/4u}\,. \end{array}$$

\end{prop}
\begin{proof}  By \cite[Theorem\,2.30]{BGV}, and
    by \cite[(3.11)]{K}, there exist smooth $v_\ell(\widetilde{x},
\widetilde{y})$ such that in $U$,
\begin{equation}\label{Hvj} \begin{array}{l}
    e^{-u \Delta^{\widetilde{X}}}(\widetilde{x}, \widetilde{y}) =
    H_M(\widetilde{x}, \widetilde{y})
    + R_M(\widetilde{x}, \widetilde{y}),\;\;\; \mbox{where}\\ \\
    H_M(\widetilde{x}, \widetilde{y}) = (4 \pi u)^{- (n + 1/2)}
\psi(d_M(\widetilde{x}, \widetilde{y})^2) e^{-d(\widetilde{x},
    \widetilde{y})^2/4u}
\sum_{\ell = 0}^{M} u^\ell v_\ell(\widetilde{x}, \widetilde{y})
j(\widetilde{x}, \widetilde{y})^{-1/2}
  \end{array} \end{equation}
with $R_M(\widetilde{x}, \widetilde{y}) \ll  (4 \pi u)^{- (n + 1/2)}
e^{-d(\widetilde{x}, \widetilde{y})^2/4u} u^{M + 1}$.
The $v_\ell$ solve the transport equations,
$$ v_\ell(\widetilde{x}, \widetilde{y}) = - \int_0^1 s^{\ell - 1}
(B_{\widetilde{x}} v_{\ell - 1})(\widetilde{x}_s, \widetilde{y}) ds.$$
Here, $B = j^{\half} \circ \Delta^{\widetilde{X}} j^{- \half}.$
Now put $\widetilde{y} = \widetilde{x} +i \theta$ and put
$\Phi_\ell(\widetilde{x}, i \theta)
= v_\ell(\widetilde{x}, \widetilde{x} + i \theta)$.
The stated estimate follows from the off-diagonal estimates
\eqref{NEAR} of \cite{K}.
 \end{proof}
We now project the heat kernel on $\widetilde{X}$ to $X$ to get:

\begin{prop} \label{AUTO} The heat kernel on $X$ is given by
\begin{equation}
e^{- u \Delta^X}(x, e^{i \theta} x) = \sum_{n \in \Z}
e^{- u \Delta^{\widetilde{X}}}(\widetilde{x}, \widetilde{x}
+ i\theta + i n)\,.
\end{equation}
Here, $p(\tilde{x}) = x$. Moreover,

 \begin{equation}\label{DOWN} \begin{array}{l}
     e^{-u \Delta^{X}}(x, e^{i \theta}x) =
     H_M(u, x, e^{i \theta} x) + R_M(x, \theta ), \;\;
     \mbox{where}\\ \\
     H_M(u, x, e^{i \theta} x) = (4 \pi u)^{- (n + 1/2)}\times\\ \\
 \Big[\sum_{\ell = 0}^{M}
\sum_{n \in \Z}
e^{- (\theta + n)^2/4u}  u^\ell
\widetilde{\Phi}_\ell(\tilde{x},  i\theta + i n)
\widetilde{j} (\tilde{x}, i\theta + in)^{-1/2}\Big]\,,
\end{array} \end{equation}
and where for $\tilde{x}$ with $\pi(\tilde{x}) = x$,
$$
R_M(\tilde{x}, i \theta) \ll (4 \pi u)^{- (n + 1/2 + M)}
\sum_{n \in \Z} e^{-|\theta + n|^2/4u}. $$

\end{prop}

\begin{proof} Both statements hold because $\tilde{X} \to X$ is
    locally isometric,
and therefore the heat kernel and parametrix on $X$ are
Poincar\'e series in those on $\widetilde{X}$.
\end{proof}

\subsection{Stationary phase calculation of the asymptotics}

We now  use the heat kernel parametrix \eqref{DOWN}
to calculate the scaled heat kernel asymptotics by the stationary
phase method.  Our calculation of the coefficients is based on
Theorem 5.8 of \cite{BGV}. We therefore rewrite Proposition
\ref{AUTO} in the form stated there.

\begin{rem} The notation in \cite{BGV} for the
    `Hadamard' coefficients $\Phi_j$ is somewhat different in their
Theorems 2.26 and 5.8. In the latter, the volume half-density
factor in the $q_t(x, y)$ factor in \cite[Theorem 2.26]{BGV}
is absorbed into the $\Phi_j$ of \cite[Theorem 5.8]{BGV}, and therefore
$\Phi_0(y,y)$ changes from $I$ to $\det^{-\half} (J(x,a))$
in \cite[Theorem 5.8]{BGV}.  Since we are using their computation of the
heat kernel expansion, we follow their notational
conventions. \end{rem}

We thus rewrite Proposition \ref{AUTO} in the form of
of \cite[Theorem 5.8]{BGV} and combine with  \eqref{FCb} to obtain,
\begin{equation} \label{BGV} \exp(- t \nabla_p^* \nabla_p)(z, z)
    \sim \Big(\frac{p}{4 \pi u}\Big)^{\!n + \frac12} e^{p u}
\int_{\R}^{asympt}
\sum_{j = 0}^{\infty} p^{-j} u^j \Phi_j(x,  i\theta)
e^{-p \theta^2/4u} e^{- i p \theta}   d\theta, \end{equation}
where (cf. \eqref{JDEF})
\begin{equation}\label{PHI0}
    \Phi_0(x, a) =( \det \;  J(x, a))^{-\half} \end{equation}
and where $\int_{\R}^{asympt}$ is the notation of  \cite[(5.5)]{BGV}
for the asymptotic expansion of the integral.
The  integral is an oscillatory integral with complex phase
\begin{equation} \label{phase}
    - |\theta|^2/4u - i \theta,\end{equation}
with  a single non-degenerate critical point at
$  \theta = - 2 u i $
with constant Hessian.  We would like to apply the method
of stationary phase to the integral,
but  the critical point is complex and in particular does
not lie in the contour of integration. This is not surprising:
the integral must be exponentially decaying
to balance the factor of $e^{p u} $ in front of it. Therefore,
we must deform the contour to $|z| = 2t.$ However, then we no
longer have the heat kernel in the real domain,
but rather the analytic continuation of the heat kernel of
$e^{- t \Delta^X}$  in the fiber direction.  Thus we first
need to discuss the analytic continuation of the heat kernel
in the fiber.

 \subsection{Analytic continuation of heat kernels}

In this section, we  analyze  the analytic continuation of the kernel
$e^{-u \Delta^X}(e^{i \theta} x, x)$ and its Hadamard parametrix
in the $e^{i \theta}$ variable
from $S^1 \times X$ to $\C^* \times X.$ Despite the fact that
the metric $h^L$ is only $\cC^{\infty}$
and not real analytic, the heat kernel always has an analytic
continuation in the variable $e^{i \theta}$,
as the next Proposition shows.

  The $S^1$ action $e^{i \theta} \cdot x$ on $X$ extends
  to a holomorphic
action of $\C^*$ on $L^*$ which we denote by $e^z \cdot \mu$
for $\mu \in L^*$. When
$\mu = x \in X$ we denote it more simply by $e^z \cdot x.$
We also write $z = t + i \theta$ with $t = 2u$
when the heat kernel is at time $u$.

\begin{prop} \label{BGVprop} The kernel
    $e(u, x, \theta): = e^{- u \Delta^X} (x, e^{i \theta} x)$
on $X \times {\bf g}$ extends for each $(u, x) \in \R^+ \times X$
to an entire function
$e(u, x, z): =  e^{- u \Delta^X} (x, e^{z} x)$ on ${\bf g}_{\C},$
with $ {\bf g}= {\rm Lie}(S^1)\simeq \mathbb{R}$.
\end{prop}

\begin{proof}  This is most easily seen using the
    Fourier/eigenfunction expansion of the heat kernel,
\begin{equation}
    e^{- u \Delta^X}(x,y) = \sum_{p \in \Z} \sum_{j = 1}^{\infty}
 e^{- \lambda_{p j} u} \phi_{p j}(x) \bar{\phi}_{p j}(y).
\end{equation}
Thus
\begin{equation} e(u, x, \theta)
    =  e^{- u \Delta^X}(x, e^{i \theta} x) =
 \sum_{p \in \Z}  \sum_{j = 1}^{\infty}e^{i p \theta}
 e^{- \lambda_{p j} u} |\phi_{p j}(x)|^2,
\end{equation}
hence the analytic continuation must be given by
\begin{equation} e^{- u \Delta^X}(x, e^z x) =  \sum_{p \in \Z}
 \sum_{j = 1}^{\infty} e^{p z}
 e^{- \lambda_{p j} u} |\phi_{p j}(x)|^2.  \end{equation}
The only question is whether the sum convergences uniformly
to a holomorphic function
of $z$.  As above, we write
$\Delta^X = \Delta_H + \frac{\partial^2}{\partial \theta^2}$.
Since $\big[\Delta_H ,\frac{\partial^2}{\partial \theta^2}\big] = 0$,  the eigenvalues
have the form
$\lambda_{p j} = p^2 + \mu_{p j}$ with $\{\mu_{p j}\}
\subset \R^+$ the spectrum
of the horizontal Laplacian $\Delta_H$ on the space
$L^2_p(X)$. 
Thus
\begin{equation} e^{- u \Delta^X}(x, e^z x) =  \sum_{p \in \Z}
  e^{p z} e^{- p^2 u} e^{- u L^p}( x, x)\,,  \end{equation}
where $e^{- u L^p}( x, y):= \int_{S^1} e^{- i p \theta}
e^{- u \Delta_H}(x, e^{i \theta} x) d\theta$.  But
$$|e^{- u L^p}( x, x)| = \Big|\int_{S^1} e^{- i p \theta}
e^{- u \Delta_H}(x, e^{i \theta} x) d\theta\Big|
\leq \int_{S^1}  e^{- u \Delta_H}(x, e^{i \theta} x) d\theta
= e_0(u,x)\,,$$
where $e_0(u,x)$ is a continuous function of $(u,x).$
The proposition follows from the fact that
$\sum_{p \in \Z}e^{p z} e^{- p^2 u}$
convergences uniformly on compact sets in $|z|$.
\end{proof}

\subsection{Analytic continuation of parametrices}

Next we consider the analytic continuation of the parametrix.
 We first observe that the connection $\beta$  extends
to $L^*$ by the requirement that  it  be $\C^*$ invariant. Thus,
$T_{\ell} L^* = H_{\ell} \oplus V_{\ell}$ where
$V_{\ell} = \C \frac{\partial}{\partial \lambda}$ where
$\lambda \cdot x$ denotes the $\C^*$ action. Using the metric
$G$ we may identify $T X$ with $T^*X$ and similarly decompose
$T^*X$ and $T^*L^*$ into horizontal and vertical spaces.
The  vertical
space $V_{\ell}^*$ is spanned by $ \alpha_{\ell}.$
\begin{prop}

The fiber distance squared function
$d^2(\tilde{x}, \tilde{x} + i \theta)$ admits an analytic extension
in $\theta$ to $\C$ satisfying
\begin{equation}
d(\widetilde{x}, \widetilde{x} + i \theta + i \lambda)^2
= (i\theta + i \lambda)^2. \end{equation}
Moreover, the Hadamard coefficients
$v_j(\tilde{x},  \tilde{x} + i\theta)$  \eqref{Hvj} admit holomorphic
extensions to $i\theta + i \lambda$.

 \end{prop}
\begin{proof}
As mentioned above,  the holomorphic continuation
of the $S^1$ action is the action of $\C^*$ on $T^* L^*$.
The first statement about the distance function is obvious
since the distance squared function on the fiber is real analytic
(this is why we lifted the heat kernel from $X$ to
$\widetilde{X}$).


The second statement  is also obvious for $v_0 = 1$.
We then prove it for the
higher $v_k$'s inductively, using the formula
\begin{equation}\label{HADA} v_{k + 1}(x, y)  =
 \int_0^1 s^k B_x v_k (x_s, y) ds \end{equation}
The geodesic $x_s$ from $x$ to $y$ stays in the `domain of
holomorphy'. Moreover,
$B_x = j^{\half} \Delta^X j^{- \half}$ so it suffices to
show that $\Delta^X$ admits
a fiberwise holomorphic continuation.
But clearly, the fiber analytic extension of   $\Delta^X$ is
$\Delta_H + (\lambda \frac{\partial}
{\partial \lambda})^2.$  Hence $v_{k+1}(x, e^z y)$
is well-defined and holomorphic in $z$. Since $j \not= 0$
for such $(x,y)$ it possesses a holomorphic
square root and inverse.

\end{proof}
\begin{cor}  \label{ACPAR}The functions $ \Phi_i(x, \theta) $ on
    $\widetilde{X} \times \R$ extend to holomorphic functions
    $\Phi_i(x, z)$ on $X \times \C $.
Since $e^{- u \Delta^X}(x, e^{i \theta} \cdot x)$ and
its parametrices admit analytic
continuations, it follows that the remainder
$R_M(u,x,  e^{i \theta} \cdot x)$ admits
an analytic continuation.\end{cor}

\begin{rem} The analytic continuation of
    $d^2(\tilde{x}, \tilde{y} +i  \theta)$ exists as long as
    $(\pi(\tilde{x}), \pi(\tilde{y}))$
is sufficiently close to the diagonal. Although we do not use
this in Theorem \ref{ASYMX} when
$\pi(\tilde{x}) \not= \pi(\tilde{y})$,
we briefly go through the proof.

First, we can holomorphically continue the geodesic
flowon $S^* X$  in the fiber variable to $S^* L^*$  by observing
that the Hamiltonian continues to
$|\xi_H|^2 + p_{\lambda}^2$ where $p_{\lambda}$ is the
analytic continuation of $p_{\theta}.$ We denote by
$\widetilde{G}^u$ the continuation
of the geodesic flow of $(X, G)$ to $L^*$.

If $x \in X $ and if $\xi = \xi_H + \lambda \beta$ (with
$\lambda \in \C^*)$,  then
$\widetilde{G}^u(x, \xi) = V^u \circ \widetilde{G}_H^u(x, \xi_H
+ \lambda \beta) = e^{u \lambda} G_H^u(x, \xi_H).$
If we denote the projection to  $X$ of
$G^t_H(x, \xi_H)$ by  $\widetilde{\exp}_{H x}(u \xi_H)$,
then we have
 $\widetilde{\exp}_x (\xi)
 = e^{\lambda} \widetilde{\exp}_{H x}( \xi_H) $.
 We note that
$\widetilde{\exp}_{H x}(u \xi_H)$ is a horizontal curve in $X$
and that
$V^u(\lambda \beta)$ is a vertical curve; thus,
$d\, \widetilde{\exp}:
H \oplus V \to H \oplus V$ is diagonal.  Also,
$G^t_H(\widetilde{x}, \xi_H)$ is the horizontal lift to $\widetilde{x}$
of its projection to
$(M,g^{TM})$.
It follows that $\widetilde{\exp}$ is non-singular in a product
neighborhood of the form
$\pi \times \pi^{-1}(U)$ where $\pi: L^* \to M$ and where
$U \subset M \times M$ is
a neighborhood of the diagonal.
Hence, for any points $\ell, \ell' \in L^*$ with $\pi(\ell), \pi(\ell')$
close enough to the diagonal in $M \times M$,
there exists a unique element $\nu \in T^*_{\ell} L^*$
with $|p_z| \leq \varepsilon $ so that
$\widetilde{\exp}_{\ell} \nu = \ell '.$
The latter is the minimizing geodesic from
$\pi(\widetilde{x})$ to $\pi(\widetilde{y})$.
Since $\exp_{\widetilde{x}} \xi
= e^{\lambda} \exp_{H \widetilde{x}} \xi_H$
we have
\begin{equation}
d(\widetilde{x}, e^{\lambda} \exp_{H \widetilde{x}} \xi_H)^2
= d(\pi(\tilde{x}), \pi( \exp_{H \widetilde{x}} \xi_H))^2
- \lambda^2.
\end{equation}

\end{rem}
\bigskip

The main lemma in the proof of Thereom \ref{ASYMX} is the
following expression for
the scaled, analytically continued heat kernel.
Note that the parameter $u$ appears twice: once as the time
parameter and once as the dilation factor in $L^*$.

\begin{lemma} \label{REMLEM}
$$\begin{array}{lll} e^{- (u/p) \Delta^{\tilde{X}}}
(\tilde{x}, \tilde{x} +  \theta -  2 iu)
& = &  C_n \Big(\frac{p}{u}\Big)^{\frac{\dim X}{2}}
e^{p (i \theta + 2  u)^2/4u}
\sum_{k = 0}^M \left(\frac{u}{p}\right)^k
\widetilde{\Phi}_k (\tilde{x}, i \theta + 2 u ) \\&&\\&&
+ R_M(u/p, \tilde{x},  i\theta - 2 i u + \tilde{ x}) \end{array}$$
where
$ R_M(u/p, \tilde{x},  i\theta - 2i u + \tilde{ x})  = O
\Big(e^{  - p \theta^2/ u }
\left(\frac{u}{p}\right)^{- \frac{\dim X}{2} + M} e^{p u}\Big)$.
\end{lemma}

\begin{proof}
We first consider the remainder in the real domain.
For simplicty we write points of $\tilde{X}$ as $x$ rather than
$\tilde{x}$. The first goal is to obtain a Duhamel type formula
(see \eqref{RMFORMULA}) for $R_M$. The derivation of this
formula  is valid for Laplacians on all Riemannian manifolds $(X, G)$,
and we therefore use the general notation $\Delta$ for a Laplacian
and $H_M$ for the Mth Hadamard parametrix.

We first note that the remainder
\begin{equation}
    R_M(u, x, y) : =   e^{- u \Delta} (x, y) - H_M(u, x, y)
\end{equation}
 solves the initial value problem
\begin{equation}\left\{ \begin{array}{l}
    (\frac{\partial}{\partial u} - \Delta)
R_M(u, x, y) = A_M(u, x, y) +  B_M(d \psi, x, y), \\ \\
R_M(0, x, y) = 0. \end{array} \right. \end{equation}
where
\begin{equation}
A_M(u, x,y) = (4 \pi u)^{- (\frac{\dim X}{2} + M)}
\psi(d_M(x, y)^2) e^{- d(x, y)^2/4u}  j(x, y )^{-1/2}
\Delta_{x} u_M(x,y)  \end{equation}
and where $B_M(d \psi, x, y)$ is the sum of the terms
in which at least one derivative falls on $\psi$. Put
\[
G(u, x, y) = \phi(d_M(x, y)^2) \frac{1}{(4 \pi u)^{\frac{\dim X}{2}}}
\, e^{-d(x, y)^2/4u} j(x, y)^{- \half}
\]
where $\phi$ is supported in a neighborhood of the diagonal, with
$\phi \equiv 1$ on supp$\psi$, and put
\begin{equation}
    R_M(u, x, y) = G(u, x, y) S_M(u, x, y). \end{equation}
The equation for $R_M$ then becomes
\begin{equation}\left\{ \begin{array}{l}
    G(u, x, y)^{-1} (\frac{\partial}{\partial u} - \Delta^X)
G(u, x, y) S_M(u, x, y) = u^M \big(\psi(d_M(x,y)^2)
\Delta^{X}_{x} v_M(x,y) + b_M(d \psi, x)\big), \\ \\
S_M(0, x, y) = 0. \end{array} \right. \end{equation}
One easily calculates (cf.\ \cite[Proposition\,2.24]{BGV})
 that
 $$G(u, x, y)^{-1} \Big(\frac{\partial}{\partial u} - \Delta^X\Big)
G(u, x, y) = \frac{\partial}{\partial u} +  u^{-1}\nabla_{\rcal}
+  j^{\half} \Delta j^{- \half}.$$
Here, $\nabla_{\rcal}$ is the directional derivative along the
radial vector field from $x$.

Multiplying through by $u$ to regularize the equation,
and changing variables to $t = \log u$, we get
\begin{equation}\left\{ \begin{array}{l}
    (\frac{\partial}{\partial t} + \nabla_{\rcal}  + t B_x) S_M(t, x, y)
    = e^{M t} (\psi(d_M(x,y)^2) \Delta_{x} v_M(x,y)
    + b_M(d \psi, x)), \\ \\
S_M(- \infty, x, y) = 0. \end{array} \right. \end{equation}
The solution is given by
\begin{multline}
 S_M(u, x, y) = \int_0^u \int_X e^{-(u - s) \Delta^X}(x, a)
 A_M(s, a, y) dv (a) ds \\+
 \int_0^u \int_X e^{-(u - s) \Delta^X}( x, a)  B_M(s, a, y) dv (a) ds.
\end{multline}

We now specialize to $\tilde{X}$ and $\Delta^{\tilde{X}}$ or
equivalently $X$ and $\Delta^X$. Our goal is
to estimate the analytic contiuation of the remainder. When dealing
with the parametrix, it is convenient to work on $\tilde{X}$  since its
distance-squared function is real analytic along the fibers.
When estimating the remainder $R_M$ it is
convenient to work on $X$ because it is compact.

In the
case of $X$
 we obtain the Duhamel type formula,
\begin{multline}\label{RMFORMULA}
    R_M(u, x, e^{i \theta} x) = \int_0^u \int_X e^{-(u - s)
    \Delta^X}( x, y)  A_M(s, y, e^{i \theta} x) dv (y) ds \\ \\+
 \int_0^u \int_X e^{-(u - s) \Delta^X}( x, y)
 B_M(s, y, e^{i \theta} x) dv (y) ds. \end{multline}
The same formula is valid on $\tilde{X}$ but there we write the
$\R$ action additively.

We observe that since $\psi(d_M(x, y)^2)$ is constant along
the fibers of $X \to M$,
both  $A_M(s, y, e^{i \theta} x)$ and  $B_M(s, y, e^{i \theta} x)$
admit  holomorphic continuations in the
variable $e^{i \theta}.$ As above, we continue to $e^z$ with
$z = 2u + i \theta$  for the heat
kernel at time $u$.   For instance,
\begin{multline}   A_M(u, y,e^z x)\\
    =  (4 \pi u)^{- (\frac{\dim X}{2} + M)}
\sum_{n \in \Z}\psi(d_M(\widetilde{x}, \widetilde{y} + n)^2)
e^{ d(\widetilde{y}, e^{z} \widetilde{x}+ n)^2/4u}
\Delta^{\widetilde{X}}_{\widetilde{y}} v_M(\widetilde{y},
e^z \widetilde{x} + n)
j(\widetilde{y}, e^z \widetilde{x}+ n)^{-1/2}\,.
\end{multline}
It follows that the analytic continuation of
$R_M(u, x, e^{i \theta} x)$ may be expressed as
\begin{equation} \begin{split}
    R_M(u, x, e^{z}x) = \int_0^u \int_{X} e^{-(u -s)\Delta^X}( x, y)
    A_M(s,y , e^{z} x) dv (y) ds\\ \\ +
 \int_0^u \int_X e^{-(u - s) \Delta^X}( e^z x, y)
 B_M(s, y,  x) dv (y) ds\,. \end{split} \end{equation}
Dilating the time variable and setting $z = 2 u +  i \theta$ gives
\begin{equation} \label{TERMS} \begin{split}
    R_M(u/p, x, e^{i \theta + 2u}x) = \frac{1}{p} \int_0^{u}
    \int_X e^{-((u - s)/p) \Delta^X}( x, y)
    A_M(s/p ,y , e^{i \theta + 2u} x) dv (y) ds\\ \\ +
 \int_0^{u/p} \int_X e^{-((u - s)/p) \Delta^X}(  x, y)
 B_M(s/p, y,  e^{i \theta + 2u} x) dv (y) ds\,.
\end{split} \end{equation}
The desired estimate on $R_M$  would follow if we could establish
that
\begin{equation} \label{Est} \begin{split}
    (i)\:  & \Big| \int_0^{u} \int_X    e^{-((u - s)/p)
    \Delta^X}( x, y)  A_M(s/p ,y , e^{i \theta + 2u} x) dv (y) ds
    \Big| \ll \left(\frac{u}{p}\right)^{\!M + 1} e^{p u}, \;\;
    \mbox{and} \\ & \\ (ii)\: &
 \Big|\int_0^{u} \int_X    e^{-((u - s)/p) \Delta^X}(  x, y)
 B_M(s/p, y,  e^{i \theta + 2u} x) dv (y) ds
 \Big| \ll \left(\frac{u}{p}\right)^{\!M + 1} e^{p u}.
\end{split} \end{equation}
We establish (\ref{Est} (i)) using
 the explicit Gaussian formula
\begin{multline}
    A_M(s/p, y,e^{i \theta + 2u} x)\\
    =  4 \pi \left(\frac{s}{p}\right)^{- (\frac{\dim X}{2} + M)}
\sum_{n \in \Z}   e^{p d(y, e^{i \theta + 2u} x + n)^2/4s}
\Delta^{\widetilde{X}}_{\widetilde{y}}
v_M(\widetilde{y}, e^{i \theta + 2u} \widetilde{x} + n)
j(\widetilde{y}, e^{i \theta + 2u} \widetilde{x}+ n)^{-1/2}
\end{multline}
and the Gaussian upper bound
\begin{equation} e^{-u \Delta^X}(u, x, y) \leq G(u,x, y),
\end{equation}
of Kannai \cite{K}.  They give that (\ref{TERMS} (i)) is bounded by
\begin{equation}\label{INEQone} \begin{split}
\ll \Big(\frac{1}{p}\Big)^{- \frac{\dim X}{2} + M + 1}
\sup_{(x,y) \in X \times X}\big|\Delta^{\widetilde{X}}_{\widetilde{y}}
 v_M(\widetilde{y}, e^{i \theta + 2u}\widetilde{x})\big| \\
\int_0^u \int_X s^{ M} G((u - s)/p, x, y)
|G(s/p, y, e^{-i \theta + 2u}  x) | dv (y)ds.\end{split} \end{equation}
Here,   $G(s, x, e^{\bar{z}}  y) = s^{- \frac{\dim X}{2} }
e^{- d(x, e^{\bar z} y)^2/ 4s}$. Its
modulus is then equal to
$|G(s, x, e^{\bar{z}}  y)| = s^{- \frac{\dim X}{2} }
e^{- \Re\, d(x, e^{\bar z} y)^2/4s}$. We can asymptotically
estimate the resulting integral
$$ \int_X \exp\!\Big(\!-p\big( d(x, y)^2/4(u - s)
+ \Re\, d(x, e^{\overline{i\theta + 2u }} y )^2/ 4s\big)\Big) dv (y)$$
by the stationary phase method.   We have,
\begin{equation} \Re\, d(x, e^{i \theta + 2t} y )^2
    = d(x,e^{i \theta} y)^2 - 4t^2. \end{equation}
Hence critical points occur at $y$ such that
\begin{equation}\begin{array}{l}
    \nabla_y d(x, y)^2/4(u - s) = - \nabla_y
    d(y, e^{- i \theta} x)^2/ 4s. \end{array} \end{equation}
Now $ \nabla_y d(x, y)^2$ is tangent to the geodesic from $x$
to $y$ and
$\nabla_y  d(y, e^{- i \theta} x) )^2$ is tangent to the geodesic
from $y$ to $e^{i \theta} x.$
Since they are multiples, it follows that $y$ must lie along
the minimizing geodesic from
$x$ to $e^{i \theta} x.$  This is just the curve
$\gamma(u) = e^{i u } x, u \in [0, \theta].$ Moreover,
$ d(x, y)/(u - s) =   d(y, e^{- i \theta} x) )/s$.  Hence we have
$$u/(u - s) = ( \theta - u)/s
\iff u \Big( \frac{1}{u - s} + \frac{1}{s}\Big) = - \frac{\theta}{s}
\iff u = - \frac{(u - s)}{u} \theta.  $$
Hence the critical locus is given by:
$y_s = e^{i \frac{(u - s)}{u} \theta} x.$ The value of the
phase along the critical locus equals
\begin{equation}\begin{array}{l}
    d(x, e^{i \frac{(u - s)}{u} \theta} x)^2/4(u - s) +d(x,e^{- i \theta}
e^{i \frac{(u - s)}{u} \theta} x)^2/ 4s - 4u^2/ 4s\\ \\
= d(x, e^{i \theta} x)^2/ 4u - 4u^2/4s =
\dfrac{\theta^2}{u} -  \dfrac{u^2}{s}\,\cdot \end{array} \end{equation}
Also, the  transverse Hessian of the phase equals
$\frac{1}{u - s} + \frac{1}{s} = \frac{u}{s(u -s)}$\,.
Raising it to the power $-\frac{\dim X}{2}$ cancels
the factors of $s^{- \frac{\dim X}{2} }$ and
$(u - s)^{- \frac{\dim X}{2}}$ and leaves $u^{- \frac{\dim X}{2}}.$
Hence
\begin{equation} \begin{split}
    \int_0^u \int_X G((u - s)/p, x, y) s^{M}
    |G(s/p, y, e^{-i \theta + 2u}  x) | dv (y)ds \\ \\ \sim
 u^{- \frac{\dim X}{2}}  e^{  - p \theta^2/ u } \int_0^u
 e^{   p  u^2/s } s^{ M}   ds \\ \\
\ll  u^{- \frac{\dim X}{2}}  e^{  - p \theta^2/ u }
e^{   p  u  } u^{M+1}.
\end{split} \end{equation}
This completes the proof of Lemma \ref{REMLEM}.
\end{proof}

\subsection{Completion of proof of Theorem \ref{ASYMX}}

We now complete the proof of Theorem \ref{ASYMX}.
We begin with the oscillatory integral \eqref{DOWN}
with complex phase \eqref{phase}
with  a single non-degenerate critical point at
$  \theta = - 2 u i $ and
with constant Hessian.
Since the critical point is complex,
 we deform the contour to $|z| = 2t.$
Thus, we have (with $z = e^{i \theta} \in S^1$)
\begin{equation}\begin{split}
    e_p^H(u,x) &:= e^{u p} \frac{1}{2 \pi i} \int_{|z|=1}
e^{-(u/p) \Delta^X}(z x, x) z^{-p}\, \frac{dz}{z} \\
&= e^{u p} \frac{1}{2 \pi i} \int_{|z|= e^{ 2 u}}
e^{-(u/p) \Delta^X}(z x, x) z^{-p}\, \frac{dz}{z}\\
&= e^{u p} \int_0^{2 \pi} e^{-(u/p)
\Delta^X}(e^{i \theta + 2u} x, x) e^{ -p ( i \theta + 2u)} d\theta.
\end{split} \end{equation}
We now plug in the Poincar\'e series formula of
Proposition \ref{AUTO} and unfold the sum over $\Z$ to get
\begin{equation}
    e_p^H(u, x) = e^{u p}  \int_{\R}  e^{-(u/p)
\Delta^{\tilde{X}} } (e^{i \theta + 2u} x, x)
e^{ -p ( i \theta + 2u)} d\theta. \end{equation}
We then substitute the parametrix for $ e^{-(u/p)
\Delta^{\tilde{X}} } (e^{i \theta + 2u} x, x)$  with remainder from
Lemma \ref{REMLEM} for
$e^{-(u/p) \Delta^X}(e^{i \theta + 2u} x, x)$.
Using the notation of \cite{BGV}, we obtain,
\begin{equation}\begin{split}
&e_p^H(u, x)  \\
&= \Big(\frac{p}{4 \pi u}\Big)^{n + \frac12}
e^{p u}\int_{\R}^{asympt}
\Big(\sum_{\ell = 0}^{M} p^{-\ell} u^\ell \Phi_\ell(x,  i\theta + 2u)
e^{-p (\theta - 2 u i)^2/4u}
e^{- p (i \theta + 2 u )} + R_M\Big) d\theta   \\
&=  \Big(\frac{p}{4 \pi u}\Big)^{n + \frac12} \int_{\R}^{asympt}
\Big(\sum_{\ell = 0}^{M} p^{-\ell} u^\ell \Phi_\ell(x, i \theta + 2u)
e^{-p \theta^2/4u} + R_M\Big)
  d\theta  .  \end{split} \end{equation}
The integral is now a standard Gaussian integral with complex phase
\eqref{phase}
\begin{equation} \label{phase2}
    - (\theta - 2 u i)^2/4u- (i \theta + 2 u ), \end{equation}
which has a unique   critical point on the line of integration
at $\theta = 0.$ We may neglect the remainder term if
we only want to expand to
order $M$ in $p^{-1}$  and apply the method of stationary phase
(see \cite[Theorem 7.7.5]{H}) to obtain,
\begin{equation} \begin{split} e_p^H(u,x)
&\sim  \Big(\frac{p}{4 \pi u}\Big)^{n + \frac12}\,
|p/u|^{-\half} \sum_{k = 0}^{M - 1}
\sum_{\ell = 0}^{\infty} p^{-\ell - k} u^\ell
\frac{1}{k!} \Big[\frac{\partial}{\partial \theta}\Big]^{2k}
  \Phi_\ell(x,  i\theta + 2 u )\Big|_{\theta = 0}\\
&= \Big(\frac{p}{4 \pi u}\Big)^{\!n}  \sum_{k = 0}^{M - 1}
\sum_{\ell = 0}^{\infty} p^{-\ell - k} u^\ell
\frac{1}{k!} \Big[\frac{\partial}{\partial \theta}\Big]^{2k}
  \Phi_\ell(x,  i\theta + 2 u )\Big|_{\theta = 0}\,. \end{split}
\end{equation}
All the properties stated in the theorem follow from standard facts
about the
stationary phase method and from the properties of
the coefficients $\Phi_j$ of \cite[Theorem 5.8]{BGV}.
By \eqref{BGV},  the principal term is given by
$\Big(\dfrac{p}{4 \pi u}\Big)^{\!n} \Phi_0(x, 2 u ) $ or equivalently,
by  using \eqref{PHI0} and Proposition \ref{JDEF},
\begin{equation}\label{PRINb}\begin{split}
    \Big(\frac{p}{4 \pi u}\Big)^{\!n} \Phi_{0}(x, 2 u )
= \Big(\frac{p}{4 \pi u}\Big)^{\!n} (\det J(x, 2 u ) )^{-\half}
= \Big(\frac{p}{4 \pi u}\Big)^{\!n}
\det\left( \dfrac{1 - e^{-  u \tau(\omega_x )}}{u \tau(\omega_x  )}
\right)^{-\half}
\end{split}
\end{equation}
This is compatible with \eqref{eq:z.3b} because for the determinant
of functions of
$\tau(\omega_x)$ on $TM$
we have
\begin{equation}\label{PRINc}
\begin{split}
\det\left( \dfrac{1 - e^{-  u \tau(\omega_x )}}
{u \tau(\omega_x  )} \right)^{-\half}&=
\det\left( \dfrac{e^{u \tau(\omega_x )/2}
- e^{-  u \tau(\omega_x )/2}}
{u \tau(\omega_x  )} \right)^{-\half}\\
&=
\det\big|_{T^{(1,0)} M}\left( \dfrac{u \tau(\omega_x  )}
{e^{u \tau(\omega_x )/2} - e^{-  u \tau(\omega_x )/2}} \right),
\end{split}
\end{equation}
and because the factor of $2$ in \eqref{taudef}
is not used in the definition of $\dot{R}^L$.
To explain the last equality and to clarify the notation 
between \eqref{PRINb}--\eqref{PRINc} and \eqref{eq:z.3}, 
we recall that $ \tau(\omega_x) = 2\dot{R}_x^L.$
If we diagonalize $\dot{R}_x^L$ as in  \eqref{lm4.4}  
as an element of
$\End(TM)$,  then under the decomposition
$TM\otimes_{\R} \C = T^{(1,0)} M \oplus T^{(0,1)} M,$
$\dot{R}_x^L = \mbox{ diag}(a_1,\ldots,a_n,  -a_1,\ldots, -a_n)$. 
Hence $\det (e^{-u\tau(\omega_x)/2}) =1$.
We refer \cite[p. 152]{BGV}  for a similar calculation.

The  subleading term is given by \eqref{SUBPRIN}.
This completes the proof of Theorem \ref{ASYMX}.
\begin{rem}
In the K\"ahler case and for a quantum line bundle $L$, i.\,e., if
$\Theta=\omega= \frac{\sqrt{-1}}{2\pi}R^L$,
then a precise formula for $e_{\infty\, 1}(u,x)$ in terms of curvature
was obtained by Dai, Liu and Ma in \cite[(5.14)]{DLM06}:
\begin{multline}\label{0d14}
e_{\infty\, 1}(u,x) = \frac{-u^{n-1}}{3 (1-e ^{-4\pi u})^n}
\left[\frac{u}{2}- \frac{u}{2\tanh ^2 (2\pi u)}\right.\\
\left. - \,\frac{2}{\sinh^2(2\pi u)} \Big( \frac{-3}{32\pi } \sinh (4\pi u)
+\frac{u}{8}\Big)\right]r^M_{x},
\end{multline}
with $r^M$ the scalar curvature of $(M, g^{TM})$.
If $\Theta$ and $(L,h^L)$ are arbitrary, a corresponding formula
should follow from \eqref{SUBPRIN} or from an adaptation
of \cite[(5.14)]{DLM06}, but it is certainly more involved
than \eqref{0d14}. For the calculation of the second coefficient
of the expansion of the Bergman kernel for non-positive line bundles
see \cite{WLu:13}.
\end{rem}

\subsection{\label{REMSECT} Further remarks}

The method of completing the square to convert the horizontal
Laplacian to the full Laplacian on $X$ is quite drastic because it
replaces the horizontal Brownian motion of the original
problem with the free Brownian motion on $X$.
It  is a natural question  to ask if one can improve
Theorem \ref{ASYMX} if one has parametrices for the horizontal heat kernels. The rest of the argument
would apply.

In certain model cases, Beals-Greiner-Gaveau construct parametrices
for heat kernels of sub-Laplacians \cite{BGG1, BGG2}.  In the case
of a positive line bundle, there should exist a
parametrix locally modeled on that of the Heisenberg group,
although we are not aware of a construction at this level
of generality. Even so it would not be useful for the main problems
of this article, i.\,e., for Hermitian
line bundles which are not positive. In the case of
a positive line bundle one can
construct a parametrix for the Szeg\"o kernel directly
(see  \cite{BSj}; see also \cite{MM07} for results on the relation
between heat kernels and Szeg\"o kernels).
In more general cases, it seems that the heat kernels have
only rarely been constructed.

In situations where  one can construct  parametrices for
the horizontal heat kernels, it seems
plausible that one could
gain  better control over the $u$ dependence of the
remainder term $R_r(p, u, x)$. The original motivation of this article
was to investigate whether there exists
a joint asymptotic expansion in $(u, p)$ which would allow one
to set $u = p^{\alpha}$ or ideally $u  = c p$
in the asymptotics. One observes that the expansion occurs
in powers of $\frac{u}{p}$ and this seems to be
the natural Planck constant for the problem.
In particular, it would be natural to try to respect the  Heisenberg
scaling in which
$\frac{\partial}{\partial \theta}$ is of weight $2$.
But the coefficients and remainder we obtain by completing
the square are not functions of
$\frac{u}{p}$, and we have little control over the remainder
$R_r(p, u, x)$, which   might be of exponential growth
(or worse) in $u$.
This reflects the fact that we must analytically continue far
out into $L^*$ to make up for the brutal addition
of $(\frac{\partial}{\partial \theta})^2$. We would probably
not have to continue so far out in $L^*$ if we add the first power
$\frac{\partial}{\partial \theta}$ as the Heisenberg scaling
would suggest.

\section{\label{MMPF} Proof by localization and rescaling of the
Dolbeault-Dirac operator}

Before going further let us recall some differential-geometric
notions. Let $\nabla^{TM}$ be the Levi-Civita connection on $TM$
and $\widetilde{\nabla}^{TM}$ the connection on $TM$ defined
by $\widetilde{\nabla}^{TM}=
\nabla^{T^{(1,0)}M}\oplus\nabla^{T^{(0,1)}M}$, where
$\nabla^{T^{(1,0)}M}$ is the Chern connection on $T^{(1,0)}M$
and $\nabla^{T^{(0,1)}M}$ is its conjugate
(see \cite[(1.2.35)]{MM07}). We set
$S=\widetilde{\nabla}^{TM}-{\nabla}^{TM}$.

We denote by $\nabla^B$ the Bismut connection
\cite[(1.2.61)]{MM07} on $TM$.
It preserves the complex
structure on $TM$ by \cite[Lemma\,1.2.10]{MM07},
thus, as in \cite[(1.2.43)]{MM07},
it induces a natural connection $\nabla^B$ on
$\Lambda (T^{*(0,1)}M)$
which preserves the $\Z$-grading.
Let $\nabla ^{B,\Lambda^{0,\bullet}}$,
$\nabla ^{B,\Lambda^{0,\bullet}\otimes L^p\otimes E}$
be the connections on $\Lambda (T^{*(0,1)}M)$,
$\Lambda (T^{*(0,1)}M)\otimes L^p\otimes E$,
defined by
\begin{align}\label{alm2.32}
\begin{split}
&\nabla ^{B,\Lambda^{0,\bullet}} = \nabla^B
+ \big\langle S(\cdot)w_j,\ov{w}_j\big\rangle,  \\
&\nabla ^{B,\Lambda^{0,\bullet}\otimes L^p\otimes E}
= \nabla^{B,\Lambda^{0,\bullet}}\otimes 1
+ 1\otimes \nabla ^{L^p\otimes E} \,,
\end{split}
\end{align}
where $\{w_j\}_{j=1}^n$\index{$w_j$,$\ov{w}_j$}
is a local orthonormal frame of $T^{(1,0)}M$
(cf.\ \cite[(1.4.27)]{MM07}).

{}Let $\Phi_E$ be the smooth self--adjoint
 section of $\End(\Lambda (T^{*(0,1)} M)\otimes E)$ on $M$
 defined by
\be\label{alm4.20}
\Phi_E=  \tfrac{1}{4}r^{M} + {^c\!
\big(R^E + \tfrac{1}{2} R^{\det}\big)}
+ \tfrac{\sqrt{-1}}{2}\, {^c\!
\big(\overline{\partial}\partial \Theta\big) }
- \tfrac{1}{8} \big|( \partial- \overline{\partial} )\Theta\big|^2\,,
\ee
cf.\ \cite[(1.3.32), (1.6.20)]{MM07}. The endomorphism $\Phi_E$
appears as the difference between the Bochner Laplacian
(cf.\ \eqref{lm1.21}) associated to the Bismut connection
$\nabla^{B,\Lambda^{0,\bullet}\otimes L^p\otimes E}$ and
the Dirac operator, cf.\ \cite[Theorem 1.4.7]{MM07}:
\begin{equation}\label{alm4.21a}
D^2_p=\Delta^{B,\Lambda^{0,\bullet}\otimes L^p\otimes E}
+\Phi_E +p \, \, {^c\! R^L}\,.
\end{equation}

We start by noting the following analogue of
\cite[Proposition 1.6.4]{MM07}.
Let $f : \R \to [0,1]$ be a smooth even function such that
 \begin{align} \label{lm4.19}
f(v) = \left \{ \begin{array}{ll}  1 \quad {\rm for}
\quad |v| \leqslant  \var/2, \\
  0 \quad {\rm for} \quad |v| \geqslant  \var.
\end{array}\right.
\end{align}
For $u>0$, $a\in \C$, set
\begin{align}  \label{lm4.20}
\begin{split}
&\boldsymbol{F}_u(a)= \int_{-\infty}^{+\infty} e ^{i v a}
\exp\!\Big(\!-\frac{v^2}{2}\,\Big)
f(\sqrt{u} v) \frac{dv}{\sqrt{2\pi}}\,,\\
&\boldsymbol{G}_u(a) =\int_{-\infty}^{+\infty} e ^{i v a}
\exp\!\Big(\!-\frac{v^2}{2}\,\Big)
\big(1-f(\sqrt{u}v)\big) \frac{dv}{\sqrt{2\pi}}\,.
\end{split}
\end{align}
The functions $\boldsymbol{F}_u(a), \boldsymbol{G}_u(a)$ are
even holomorphic functions.
The restrictions of $\boldsymbol{F}_u, \boldsymbol{G}_u$ to
$\R$ lie in the Schwartz space $\mathcal{S} (\R)$. Clearly,
\begin{equation}\label{lm4.21}
\boldsymbol{F}_{u}(\upsilon D_p) +
\boldsymbol{G}_{u}(\upsilon D_p)
= \exp\!\Big(\kern-3pt-\frac{\upsilon^2}{2} D_p^2\Big).
\end{equation}

For $x,x'\in M$ let $\boldsymbol{F}_{u}(\upsilon D_p)(x,x')$,
$\boldsymbol{G}_{u}(\upsilon D_p)(x,x')$
be the smooth kernels associated to
 $\boldsymbol{F}_{u}(\upsilon D_p)$,
 $\boldsymbol{G}_{u}(\upsilon D_p)$,
calculated  with respect to the Riemannian volume form  $dv_M(x')$.
Let $B^M(x,\var)$ be the open ball in $M$
with center $x$ and radius $\var$.

\begin{prop} \label{lmt4.4} For any $m\in \N$, $T>0, \var>0$,
 there exists $C>0$ such that for any $x,x'\in M$, $p\in \N^*$,
 $0<u<T$,
\begin{eqnarray}\label{lm4.22}
\Big |\boldsymbol{G}_{u/p}(\sqrt{u/p}\,D_p)(x,x')\Big|_{\cC^m}
 \leqslant C \exp\!\Big(\kern-3pt-\frac{\var^2 p}{32 u}\,\Big).
\end{eqnarray}
Here the $\cC^m$ norm is induced by $\nabla^L, \nabla^E$,
 $\nabla^{B,\Lambda^{0,\bullet}}$ and  $h^L, h^E$, $g^{TM}$.\\
The kernel
$\boldsymbol{F}_{u/p}\big(\!\sqrt{u/p}\, D_p\big)(x, \cdot)$
only depends on the restriction of $D_p$ to
$B^M(x,\var)$, and is zero outside $B^M(x,\var)$.
\end{prop}
This follows from the proof of \cite[Proposition 1.6.4]{MM07},
in particular from \cite[(1.6.16)]{MM07} with $\zeta=1$,
since under our assumption any polynomial in $p, u^{-1}$
will be absorbed by the factor $\exp(-\frac{\var^2 p}{32 u})$.
The second assertion of follows by using
\eqref{lm4.20}, the finite propagation speed of the wave operator,
cf.\ \cite[Theorem D.2.1 and (D.2.17)]{MM07}.

Thus the problem on the asymptotic expansion of
$\exp\!\big(\!-\frac{u}{p} D_p^2\big) (x,x)$, for $0<u<T$
and $p\in\N$, is a local problem
and only depends on the restriction of $D_p$ to
$B^M(x,\var)$.

To analyze the local problem, we fix $x_0\in M$ and work on
$M_0:=\R^{2n}\simeq T_{x_0}M$.
From now on, we identify $B^{T_{x_0}M}(0,4\var)$
with $B^{M} (x_0,4\var)$ by the exponential map.
For $Z\in B^{T_{x_0}M}(0,4\var)$,
we identify
\[E_Z\cong E_{x_0}\,,\quad L_Z\cong L_{x_0}\,,\quad
\Lambda (T^{*(0,1)}_Z M)\cong \Lambda (T^{*(0,1)}_{x_0}M)\,,\]
by parallel transport with respect
to the connections $\nabla ^E$, $\nabla ^L$,
$\nabla^{B,\Lambda^{0,\bullet}}$ along the curve
$[0,1]\ni u\mapsto uZ$.
Thus on $B^{M} (x_0,4\var)$, we have the following identifications
of Hermitian bundles
\[
\begin{split}
(E, h^E)\cong (E_{x_0}, h^{E_{x_0}})\,,&\: (L, h^L)
\cong (L_{x_0}, h^{L_{x_0}})\,,\:
(\Lambda (T^{*(0,1)} M), h^{\Lambda^{0,\bullet}})\cong
(\Lambda (T^{*(0,1)}_{x_0}M), h^{\Lambda^{0,\bullet}_{x_0}})\\
&(E_p,h_p)\cong(E_{p,x_0}, h^{E_{p,x_0}})\,,
\end{split}
\]
where the bundles on the right-hand side are trivial Hermitian
bundles.

Let $\Gamma ^E, \Gamma ^L, \Gamma ^{B, \Lambda ^{0,\bullet}}$
be the corresponding connection forms of $\nabla^E$, $\nabla^L $
and $\nabla^{B, \Lambda ^{0,\bullet}}$ on $B^{M} (x_0,4\var)$.
Then $\Gamma ^E, \Gamma ^L, \Gamma ^{B, \Lambda ^{0,\bullet}}$
are skew-adjoint with respect to $h^{E_{x_0}}$,
$h^{L_{x_0}}$, $h^{\Lambda^{0,\bullet}_{x_0}}$.

Let $\rho: \R\to [0,1]$ be a smooth even function such that
\begin{align}\label{alm4.19}
\rho (v)=1  \  \  {\rm if} \  \  |v|<2;
\quad \rho (v)=0 \   \   {\rm if} \  |v|>4.
\end{align}
 Denote by  $\nabla_U$ the ordinary differentiation
 operator on $T_{x_0}M$ in the direction $U$.
{}From the above discussion,
\be\label{alm4.21}
\nabla^{E_{p, x_0}}= \nabla + \rho\!
\left(\tfrac{1}{\var}|Z|\right)\!\Big(p\, \Gamma ^L + \Gamma ^E
+\Gamma ^{B, \Lambda ^{0,\bullet}}\Big)\!(Z),
\ee
 defines a Hermitian connection on $(E_{p,x_0}, h^{E_{p,x_0}})$
on $\R^{2n} \simeq T_{x_0}M$
where the identification is given by
\begin{equation}\label{alm4.22}
\R^{2n}\ni(Z_1,\ldots, Z_{2n}) \longmapsto \sum_i
Z_i e_i\in T_{x_0}M.
\end{equation}
Here $\{e_{2j-1}= \frac{1}{\sqrt{2}}(w_{j}+\ov{w}_{j}),
e_{2j}= \frac{\sqrt{-1}}{\sqrt{2}}(w_{j}-\ov{w}_{j})\}_j$
is an orthonormal basis of $T_{x_0}M$.

Let $g^{TM_0}$ be a metric on $M_0:=\R^{2n}$ which
coincides with $g^{TM}$ on  $B^{T_{x_0}M} (0,2\var)$, and
$g^{T_{x_0}M}$ outside $B^{T_{x_0}M} (0,4\var)$.
Let $dv_{M_0}$ be the Riemannian volume form of
$(M_0, g^{TM_0})$.
Let $\Delta^{E_{p, x_0}}$ be the Bochner Laplacian associated to
$\nabla^{E_{p, x_0}}$ and $g^{TM_0}$ on $M_0$.
Set
\begin{equation}\label{alm4.23}
{\boldsymbol L}_{p,x_0}  = \Delta^{E_{p, x_0}}
- p\,\rho\!\left(\tfrac{1}{\var}|Z|\right)\!(2 \om_{d, Z} +\tau_{Z})
- \rho\!\left(\tfrac{1}{\var}|Z|\right)\! \Phi_{E,Z}.
\end{equation}
Then ${\boldsymbol L}_{p,x_0}$ is a self--adjoint operator
with respect to the $L^2$ scalar product induced by
$h^{E_{p,x_0}}$, $g^{TM_0}$ on $M_0$.
Moreover, ${\boldsymbol L}_{p,x_0}$ coincides with
$D^2_p$ on $B^{T_{x_{0}}M} (0,2\var)$.
By using \eqref{alm4.21a} we obtain the analogue of
Proposition \ref{lmt4.4} for
$\sqrt{\frac{u}{p}{\boldsymbol L}_{p,x_0}}$\,.
Thus by using the finite propagation speed for the wave operator
we get
\begin{equation}\label{alm4.24}
\left|\exp\!\Big(\!\!-\frac{u}{2p}D^2_p\Big)(x_0,x_0)
-\exp\!\Big(\!\!-\frac{u}{2p}{\boldsymbol L}_{p,x_0}\Big)(0,0)
\right|\leqslant C\exp\!\Big(\!\!-\frac{\varepsilon^2p}{32u}\Big)\,.
\end{equation}
Let $dv_{TM}$ be the Riemannian volume form on
$(T_{x_0}M, g^{T_{x_0}M})$.
Let $\kappa (Z)$ be the smooth positive function defined by
the equation
\be\label{alm4.26}
dv_{M_0}(Z) = \kappa (Z) dv_{TM}(Z),
\ee
with $k(0)=1$.

Set $\bE_{x_0}:=(\Lambda (T^{*(0,1)}M)\otimes E)_{x_0}$.
For  $s \in \cC^{\infty}(\R^{2n}, \bE_{x_0})$, $Z\in \R^{2n}$ and
$t=\frac{\sqrt{u}}{\sqrt{p}}$\,,
 set\index{$L^t_2$}
\begin{align}\label{lm4.28}
\begin{split}
&(S_{t} s ) (Z) =s (Z/t), \\
&\nabla_{t,u}=  S_t^{-1}
t \kappa^{1/2}\nabla ^{E_{p,x_0}} \kappa^{-1/2} S_t, \\
&  \cL^{t,u}_2= S_t^{-1} \kappa^{1/2} t^2 {\boldsymbol L}_{p,x_0}
\kappa^{-1/2} S_t.
\end{split}
\end{align}
Note that in \cite[(1.6.27)]{MM07} we used the scaling parameter
$t=\frac{1}{\sqrt{p}}$\,.
In the present situation we wish to obtain an expansion in the
variable $\frac{p}{u}$, so we need to
rescale the coordinates by setting $t=\frac{\sqrt{u}}{\sqrt{p}}$\,.
Put
\begin{align}\label{lm4.29}
\begin{split}
&\nabla_{0, u,\bullet}= \nabla_\bullet
+ \tfrac{u}{2} R^L_{x_0}(Z,\cdot),\\
&\cL^{0,u}_2= - \sum_i (\nabla_{0,u,e_i})^2
- 2u \om_{d,x_0} - u\tau_{x_0}\,.
\end{split}
\end{align}
Then we have the following analogue of
\cite[Theorem 4.1.7]{MM07}.
\begin{thm} \label{bkt2.5} There exist polynomials $\mA_{i,j,r}$
{\rm(}resp.\ $\mB_{i,r}$, $\mC_{r}${\rm)} in the variables $Z$
and in $u$,
where $r\in \N, i,j\in \{1,\cdots, 2n\}$,
with the following properties:

$\bullet$ their coefficients are polynomials in
$R^{TM}$  {\rm(}resp.\ $R^{TM}$, $R^{B,\Lambda^{0,\bullet}}$,
   $R^{E}$, $R^{\det}$, $d\Theta$, $R^L${\rm)}
and their derivatives at $x_0$ up to order $r-2$
{\rm(}resp.\ $r-2$, $r-2$,
 $r-2$, $r-2$, $r-1$, $r${\rm)}\,,

$\bullet$ $\mA_{i,j,r}$ is a homogeneous polynomial in $Z$ of
degree $r$
and does not depend on $u$,
the degree in $Z$ of $\mB_{i,r}$ is $\leqslant r+1$
{\rm(}resp.\ the degree of $\mC_{r}$ in $Z$ is
$\leqslant r+2${\rm)}, and has the same parity
as $r-1$ {\rm(}resp.\ $r${\rm)}\,,
the degree in $u$ of $\mB_{i,r}$ is $\leqslant 1$,
and the degree in $u$ of $\mC_{r}$ is $\leqslant 2$,

$\bullet$ if we denote by
\begin{align}\label{bk2.22}
\mO_{u,r} =  \mA_{i,j,r}\nabla_{e_i}\nabla_{e_j}
+ \mB_{i,r}(u)\nabla_{e_i}+ \mC_{r}(u),
\end{align}
 then
\begin{align}\label{bk2.23}
\cL^{t,u}_2=  \cL^{0,u}_2+ \sum_{r=1}^m t^r \mO_{u,r}
+ \cO(t^{m+1}),
\end{align}
and there exists $m'\in \N$ such that for any $k\in \N$,
$t\leqslant 1$, $0<u<T$,
the derivatives of order $\leqslant k$ of the coefficients of
the operator
 $\cO(t^{m+1})$ are dominated by $C t^{m+1} (1+|Z|)^{m'}$.
\end{thm}
Set $g_{ij}(Z)= g^{TM_0}(e_i,e_j)(Z) =  \langle e_i,e_j\rangle_Z$
and let $(g^{ij}(Z))$ be the inverse of the matrix $(g_{ij}(Z))$.
We observe that $pt = \frac{u}{t}$, thus the analogue
of \cite[(4.1.34)]{MM07} reads
\begin{equation}\label{bk2.26}
\begin{split}
\nabla_{t,u, \bullet} =&\kappa^{1/2}(tZ)
\Big(\nabla_\bullet+  t\, \Gamma ^{A_0}_{tZ}
+  \frac{u}{t}\, \Gamma ^{L_0}_{tZ}
+ t\, \Gamma ^{E_0}_{tZ}\Big)\kappa^{-1/2}(tZ),\\
\cL^{t,u}_2=& - g^{ij}(tZ) \Big(\nabla_{t,u, e_i} \nabla_{t,u, e_j}
- t\, \Gamma_{ij}^k(tZ) \nabla_{t,u, e_k}\Big)\\
&- 2 u \,\om_{0, d, tZ} - u\, \tau_{0, tZ}+ t^2\Phi_{E_0,tZ}.
\end{split}
\end{equation}
Comparing with \cite[(4.1.34)]{MM07},
the term of $\nabla_{t,u, \bullet}$ involving $u$ is
$\frac{u}{t} \Gamma ^{L_0}_{tZ}$
instead of $\frac{1}{t} \Gamma ^{L_0}_{tZ}$ therein.

Theorem \ref{bkt2.5} follows by taking  the Taylor expansion of
\eqref{bk2.26}.
Using Theorem \ref{bkt2.5} we see that
\cite[Theorems 1.6.7-1.6.10]{MM07}
 (or more precisely \cite[Theorems 4.1.9-4.1.14]{MM07} with the
 contour $\delta\cup \Delta$ replaced by the contour
 $\Gamma$ from \cite[Theorems 1.6.7-1.6.10]{MM07})
 hold  uniformly for $0<u<T$.

 Thus we get the following analogue of \cite[Theorem 4.2.8]{MM07}
 in normal coordinates.
 \begin{thm} \label{bkt3.6}  There exists $C''>0$ such that
for any $k,m,m'\in \N$, there exists
$C>0$  such that if
$t\in ]0,1],  0<u<T$, $Z,Z'\in T_{x_0}M$,
 \begin{align}\label{bk3.28}
\begin{split}
&\sup_{|\alpha|+|\alpha'|\leqslant m}
\left|\frac{\partial^{|\alpha|+|\alpha'|}}{\partial Z^{\alpha}
{\partial Z'}^{\alpha'}}\Big(\!\exp\!\big(\!-\cL^{t,u}_2\big)\!
- \sum_{r=0}^k J_{r,u} t^r\Big) (Z,Z')\right|_{\cC^{m'}(X)} \\
&\hspace*{5mm}    \leqslant C  t^{k+1}  (1+|Z|+|Z'|)^{M_{k+1,m,m'}}
\exp (- C'' |Z-Z'|^2).
\end{split}
\end{align}
\end{thm}
Note that we use the operator $\cL^{t,u}_2$ and we rescale
the coordinates by the factor $t=\frac{\sqrt{u}}{\sqrt{p}}$\,,
thus the factor $u$ in the right-hand side of the
second equation of \cite[(4.2.30)]{MM07} is $1$ here.
 Moreover, we have (cf.\ also \cite[(1.6.61)]{MM07} )
 \begin{align}\label{eq:z.1}
J_{0,u}(Z,Z')= \exp\!\big({-\cL^{0,u}_{2}}\big)\!(Z,Z') .
\end{align}
We infer from \eqref{lm4.29} (compare \cite[(1.6.68)]{MM07}) that
\begin{align}\label{lm4.76}
&\exp\!\big({-\cL^{0,u}_{2}}\big)\!(0,0) =\frac{1}{(2\pi)^{n}}
\frac{\det (u \dot{R}^L_{x_0})\exp(2u \omega_{d,x_0})}
{\det (1-\exp(-2u\dot{R}^L_{x_0}))}\otimes \Id_{E}.
\end{align}
The analogue of \cite[(1.6.66),\,(4.2.37)]{MM07} is
 that  for $Z,Z'\in T_{x_0}M$,
\begin{align}\label{abk3.45}
&\exp\!\Big(\!\!-\frac{u}{p} {\boldsymbol L}_{p,x_0}\Big)\!(Z,Z')
= \left(\frac{p}{u}\right)^{\!n}
\exp\!\big(\!-\cL^{t,u}_2\big)\!\left(\!\frac Zt,\frac {Z'}t\right)
\kappa ^{-1/2}(Z)\kappa ^{-1/2}(Z')\,.
\end{align}
 By taking $Z=Z'=0$  in Theorem \ref{bkt3.6},
 and using \eqref{abk3.45}, we get
 the analogue of \cite[(4.2.39)]{MM07},
 \begin{align}\label{bk3.46}
\abs{\del{\frac{u}{p}}^{\!n} \!
\exp\!\del{-\frac{u}{p} D^{2}_p}\!(x_0,x_0)
-  \sum_{r=0}^k  J_{r,u}(0,0)
\left(\frac{p}{u}\right)^{\!-r/2}}_{\cC^{m'}(X)}
 \leqslant C \left(\frac{p}{u}\right)^{\!-\frac{k+1}{2}}.
\end{align}
Finally, by the same argument as in the proof of
\cite[(4.2.40)]{MM07},
we get for any $r\in \N$,
\begin{equation}\label{abk3.46}
J_{2r+1,u}(0,0)=0.
\end{equation}
Relations \eqref{eq:z.1}--\eqref{abk3.46} yield
Theorem \ref{ASYM} with $e_{\infty 0}(u,x_{0})$ given by \eqref{eq:z.3}.

\providecommand{\href}[2]{#2}

\end{document}